\documentclass[10pt,a4paper]{article}

\usepackage{graphicx}
\usepackage{epstopdf}
\usepackage{epsfig}
\usepackage{fullpage} 

\usepackage[T1]{fontenc} 
\usepackage[utf8]{inputenc}

\usepackage{amsmath} 

\usepackage{amsfonts}
\usepackage{amssymb}
\usepackage{pstricks}
\usepackage{mathabx}

\newtheorem{theorem}{Theorem}[section]
\newtheorem{lemma}[theorem]{Lemma}
\newtheorem{proposition}[theorem]{Proposition}

\newtheorem{remark}[theorem]{Remark}

\newenvironment{proof}[1][Proof]{\begin{trivlist}
\item[\hskip \labelsep {\bfseries #1}]}{\end{trivlist}}
\newenvironment{definition}[1][Definition]{\begin{trivlist}
\item[\hskip \labelsep {\bfseries #1}]}{\end{trivlist}}

\newcommand{\modd}[1]{\vert #1\vert^2}
\newcommand{\nablag}{\nabla^{\gamma}}
\newcommand{\zetaa}{\zeta_{(\alpha)}}
\newcommand{\psia}{\psi_{(\alpha)}}
\newcommand{\zetak}{\zeta_{(k)}}
\newcommand{\psik}{\psi_{(k)}}
\newcommand{\psiak}{\psi_{(\alpha,k)}}
\newcommand{\zetaak}{\zeta_{(\alpha,k)}}
\newcommand{\wsurf}{\underline{w}}
\newcommand{\Vu}{\overline{V}}
\newcommand{\Vd}{\underline{V}}

\newcommand{\ds}{\displaystyle}
\newcommand{\D}{\vert D^\gamma\vert}

\newcommand{\K}{\mathcal{K}}
\newcommand{\R}{\mathbb{R}}
\newcommand{\B}{\mathfrak{P}}
\newcommand{\rt}{\underline{\mathfrak{a} }}
\newcommand{\Km}{\mathcal{K}(\epsilon\sqrt{\mu}\nablag\zeta)}
\newcommand{\Kk}{\mathcal{K}_{(k)}(\epsilon\sqrt{\mu}\nablag\zeta)}
\newcommand{\Uk}{U_{(k)}}
\newcommand{\Sy}{\mathcal{S}^1[U]}
\newcommand{\sy}{\mathcal{S}^2[U]}

\def \epsilon {\varepsilon}

\begin{document}
\title{The Cauchy problem on large time for the Water Waves equations with large topography variations}
\author{Mésognon-Gireau Benoît\footnote{UMR 8553 CNRS, Laboratoire de Mathématiques et Applications de l'Ecole Normale Supérieure, 75005 Paris, France. Email: benoit.mesognon-gireau@ens.fr}}
\date{}
\maketitle

\begin{abstract}
This paper shows that the long time existence of solutions to the Water Waves equations remains true for large amplitude topography variations in presence of surface tension. More precisely, the dimensionless equations depend strongly on three parameters $\epsilon,\mu,\beta$ measuring the amplitude of the waves, the shallowness and the amplitude of the bathymetric variations respectively. In \cite{alvarez}, the local existence of solutions to this problem is proved on a time interval of size $\frac{1}{\max (\beta,\epsilon)}$ and uniformly with respect to $\mu$. In presence of large bathymetric variations (typically $\beta \gg \epsilon$), the existence time is therefore considerably reduced. We remove here this restriction and prove the local existence on a time interval of size $\frac{1}{\epsilon}$ under the constraint that the surface tension parameter must be at the same order as the shallowness parameter $\mu$. We also show that the result of \cite{bresch_metivier} dealing with large bathymetric variations for the Shallow Water equations can be viewed as a particular endpoint case of our result.\end{abstract}

\section{Introduction}
We recall here classical formulations of the Water Waves problem, with and without surface tension. We then shortly introduce the meaningful dimensionless parameters of this problem, and then state the local existence result proved by \cite{alvarez}. We discuss the dependance of the size of the time interval with respect to these parameters and then explain the strategy adopted in this paper to get an improved local existence result. 
\subsection{Formulations of the Water Waves problem}
The Water Waves problem puts the motion of a fluid with a free surface into equations. We recall here two equivalent formulations of the Water Waves equations for an incompressible and irrotationnal fluid. We then introduce the surface tension, and recall a local existence result by \cite{alvarez}.
\subsubsection{Free surface $d$-dimensional Euler equations} 
The motion, for an incompressible, inviscid and irrotationnal fluid occupying a domain $\Omega_t$ delimited below by a fixed bottom and above by a free surface is described by the following quantities : \\\\
- the velocity of the fluid $U=(V,w)$, where $V$ and $w$ are respectively the horizontal and vertical components \\
- the free top surface profile $\zeta$ \\
- the pressure $P.$ \\
All these functions depends on the time and space variables $t$ and $(X,z) \in\Omega_t$. There exists a function $b:\mathbb{R}^d\rightarrow \mathbb{R}$ such that the domain of the fluid at the time $t$ is given by  $$\Omega_t = \lbrace (X,z)\in\mathbb{R}^{d+1},-H_0+ b(X) < z <\zeta(t,X)\rbrace,$$ where $H_0$ is the typical depth of the water. The unknowns $(U,\zeta,P)$ are governed by the Euler equations: 
\begin{align}
\begin{cases}
\partial_t V +  (V\cdot\nabla + w\partial_z)V = - \nabla P \text{ in } \Omega_t\\
\mbox{\rm div}(U) = 0 \text{ in } \Omega_t\\
\mbox{\rm curl}(U) = 0 \text{ in } \Omega_t .
\label{euler}
\end{cases}\end{align}

In these equations, $\underline{V}$ and $\underline{w}$ are the horizontal and vertical components the velocity evaluated at the surface. We denote here $-ge_z$ the acceleration of gravity, where $e_z$ is the unit vector in the vertical direction. The notation $\nabla$ denotes here the gradient with respect to the horizontal variable $X$. \\

These equations are completed by boundary conditions : 

\begin{align}
\begin{cases}
\partial_t \zeta +\underline{V}\cdot\nabla\zeta - \underline{w} = 0  \\
U\cdot n = 0 \text{ on } \lbrace z=-H_0+ b(X)\rbrace \\
P = P_{atm}\text{ on }  \lbrace z=\zeta(X)\rbrace. \label{boundary_conditions}
\end{cases}
\end{align}
The vector $n$ in the last equation stands for the normal upward vector at the bottom $(X,z=-H_0+b(X))$. We denote $P_{atm}$ the constant pressure of the atmosphere at the surface of the fluid. The first equation states the assumption that the fluid particles does not cross the surface, while the last equation states the assumption that they do not cross the bottom. The equations \eqref{euler} with boundary conditions \eqref{boundary_conditions} are commonly referred to the free surface Euler equations.
\subsubsection{Craig-Sulem-Zakharov formulation} 
Since the fluid is by hypothesis irrotational, it derives from a scalar potential: $$U = \nabla_{X,z} \Phi.$$ Here, $\nabla_{X,z}$ denotes the three dimensional gradient with respect to both variables $X$ and $z$. Zakharov remarked in \cite{zakharov} that the free surface profile $\zeta$ and the potential at the surface $\psi = \Phi_{\vert z=\zeta}$ fully determine the motion of the fluid, and gave an Hamiltonian formulation of the problem. Later, Craig-Sulem, and Sulem (\cite{craigsulem1} and \cite{craigsulem2}) gave a formulation of the Water Waves equation involving the Dirichlet-Neumann operator. The following Hamiltonian system is equivalent (see \cite{david} and \cite{alazard} for more details) to the free surface Euler equations \eqref{euler} and \eqref{boundary_conditions}:
\begin{align*}\begin{cases}
        \displaystyle{\partial_t \zeta -  G\psi = 0} \\
        \displaystyle{\partial_t \psi + g\zeta + \vert\nabla\psi\vert^2 - \frac{(G\psi +\nabla\zeta\cdot\nabla\psi)^2}{2(1+\mid\nabla\zeta\mid^2)}=0 } 
\end{cases}\end{align*}
where the unknown are $\zeta$ (free top profile) and $\psi$ (velocity potential at the surface) with $t$ as time variable and  $X\in\mathbb{R}^d$ as space variable. The fixed bottom profile is $b$, and $G$ stands for the Dirichlet-Neumann operator, that is 

\begin{equation*}
G\psi = G[\zeta, b]\psi = \sqrt{1+ \modd{\nabla\zeta}} \partial_n \Phi_{\vert z=\zeta},
\end{equation*} 
where $\Phi$ stands for the potential, and solves Laplace equation with Neumann (at the bottom) and Dirichlet (at the surface) boundary conditions 

\begin{align}\begin{cases}
\Delta_{X,z} \Phi = 0 \quad \text{in } \mathbb{R}^d \times \lbrace -H_0+ b < z < \zeta\rbrace \\
\phi_{\vert z=\zeta} = \psi,\quad \partial_n \Phi_{\vert z=-H_0+ b} = 0 \label{dirichlet}
\end{cases}\end{align}
with the notation, for the normal derivative $$\partial_n \Phi_{\vert z=-H_0+b(X)} = \nabla_{X,z}\Phi(X,-H_0+b(X))\cdot n$$ where $n$ stands for the normal upward vector at the bottom $(X,-H_0+b(X))$. See also \cite{david} for more details.
\\

\subsubsection{Water Waves with surface tension}
In the presence of surface tension, the only difference with the classical Euler problem \eqref{euler} is that the pressure condition at the surface is changed into $$P-P_{atm} = \sigma \kappa(\zeta) \text{ on } \lbrace z=\zeta(t,X)\rbrace,$$ where $\sigma$ denotes the surface tension coefficient, and $\kappa(\zeta)$ is the mean curvature at the surface $$\kappa(\zeta) = -\nabla\cdot(\frac{\nabla\zeta}{\sqrt{1+\vert\nabla\zeta\vert^2}}).$$ The Zakharov/Craig-Sulem formulation has therefore to be changed into 

\begin{align}\begin{cases}
        \displaystyle{\partial_t \zeta -  G\psi = 0} \\
        \displaystyle{\partial_t \psi + g\zeta + \vert\nabla\psi\vert^2 - \frac{(G\psi +\nabla\zeta\cdot\nabla\psi)^2}{2(1+\mid\nabla\zeta\mid^2)}=-\frac{\sigma}{\rho}\kappa(\zeta). } \label{ww_equation}
\end{cases}\end{align}

\subsubsection{Dimensionless equations}
Since the properties of the solutions depend strongly on the characteristics of the fluid, it is more convenient to non-dimensionalize the equations by introducing some characteristic lengths of the wave motion: \\\\
(1) The characteristic water depth $H_0$ \\
(2) The characteristic horizontal scale $L_x$ in the longitudinal direction \\
(3) The characteristic horizontal scale $L_y$ in the transverse direction (when $d=2$)\\
(4) The order of the free surface amplitude $a_{surf}$\\
(5) The order of bottom topography $a_{bott}$.\\

Let us then introduce the dimensionless variables: $$x'=\frac{x}{L_x},\quad y'=\frac{y}{L_y},\quad \zeta'=\frac{\zeta}{a_{surf}},\quad z'=\frac{z}{H_0},\quad b'=\frac{b}{a_{bott}},$$ and the dimensionless variables: $$t'=\frac{t}{t_0},\quad \Phi'=\frac{\Phi}{\Phi_0},$$ where $$t_0 = \frac{L_x}{\sqrt{gH_0 }},\quad \Phi_0 = \frac{a_{surf}}{H_0}L_x \sqrt{gH_0}.$$ 

After re scaling, four dimensionless parameters appear in the equation. They are 

\begin{align*}
\frac{a_{surf}}{H_0} = \epsilon, \quad \frac{H_0^2}{L_x^2} = \mu,\quad \frac{a_{bott}}{H_0} = \beta,\quad \frac{L_x}{L_y} = \gamma,\quad B_0=\frac{\rho gL_x^2}{\sigma},
\end{align*}
where  $\epsilon,\mu,\beta,\gamma$ are commonly referred  respectively as "nonlinearity", "shallowness", "topography", "transversality" and "Bond" parameters.\\

For instance, the Zakharov-Craig-Sulem system (\ref{ww_equation}) becomes (see \cite{david} for more details) in dimensionless variables (we omit the "primes" for the sake of clarity): 

\begin{align}\begin{cases}
        \displaystyle{\partial_t \zeta - \frac{1}{\mu} G_{\mu,\gamma}[\epsilon\zeta,\beta b]\psi = 0 }\\
        \ds \partial_t \psi + \zeta + \frac{\epsilon}{2}\vert\nablag\psi\vert^2 - \frac{\epsilon}{\mu}\frac{(G_{\mu,\gamma}[\epsilon\zeta,\beta b]\psi +\epsilon\mu\nablag\zeta\cdot\nablag\psi)^2}{2(1+\epsilon^2\mu\mid\nablag\zeta\mid^2)}=-\frac{1}{B_0}\frac{\kappa_\gamma(\epsilon\sqrt{\mu}\zeta)}{\epsilon\sqrt{\mu}},  \label{ww_equation1}
\end{cases}\end{align} where $G_{\mu,\gamma}[\epsilon\zeta,\beta b]\psi$ stands for the dimensionless  Dirichlet-Neumann operator, 

\begin{equation*}
G_{\mu,\gamma}[\epsilon\zeta,\beta b]\psi = \sqrt{1+\epsilon^2 \modd{\nablag\zeta}} \partial_n \Phi_{\vert z=\epsilon\zeta} =  (\partial_z\Phi-\mu\nabla^{\gamma}(\epsilon \zeta)\cdot\nabla^{\gamma}\Phi)_{\vert z=\epsilon\zeta},
\end{equation*} 
where $\Phi$ solves the Laplace equation with Neumann (at the bottom) and Dirichlet (at the surface) boundary conditions 

\begin{align*}\begin{cases}
\Delta^{\mu,\gamma} \Phi = 0 \quad \text{in } \mathbb{R}^d \times \lbrace -1+\beta b < z < \epsilon\zeta\rbrace \\
\phi_{\vert z=\epsilon\zeta} = \psi,\quad \partial_n \Phi_{\vert z=-1+\beta b} = 0,
\end{cases}\end{align*}
\\
and where the surface tension term is $$\kappa_\gamma(\zeta) = -\nablag\cdot(\frac{\nablag\zeta}{\sqrt{1+\vert\nablag\zeta\vert^2}}).$$
We used the following notations:

\begin{align*}
&\nabla^{\gamma} = {}^t(\partial_x,\gamma\partial_y) \quad &\text{ if } d=2\quad &\text{ and } &\nabla^{\gamma} = \partial_x &\quad \text{ if } d=1 \\
&\Delta^{\mu,\gamma} = \mu\partial_x^2+\gamma^2\mu\partial_y^2+\partial_z^2 \quad &\text{ if } d=2\quad &\text{ and } &\Delta^{\mu,\gamma} = \mu\partial_x^2+\partial_z^2 &\quad \text{ if } d=1
\end{align*}
and  $$ \partial_n \Phi_{\vert z=-1+\beta b}=\frac{1}{\sqrt{1+\beta^2\vert\nablag b\vert^2}} (\partial_z\Phi-\mu\nabla^{\gamma}(\beta b)\cdot\nabla^{\gamma}\Phi)_{\vert z=-1+\beta b}.$$

\subsection{Main result}

Alvarez-Lannes (\cite{alvarez}) proved the following local existence result. We use the notation $a\vee b = \max(a,b)$.
\begin{theorem}
Under reasonable assumptions on the initial conditions $(\zeta^0,\psi^0)$, there exists a unique solution $(\zeta,\psi)$ of the Water Waves equations $(\ref{ww_equation1})$ with initial condition $(\zeta^0,\psi^0)$ on a time interval $[0;\frac{T}{\epsilon\vee\beta}]$, where $T$ only depends on initial data.\label{david}
\end{theorem}
For a precise statement, see section \ref{statement} and see \cite{david} Theorem 9.6 and Chapter 4 for a complete proof. The fact that $T$ does not depend on $\mu$ allowed the authors to provide a rigorous justification of most of the Shallow Water models used in the literature for the description, among others, of coastal flows. In these models, one has $\beta = O(\epsilon)$ and therefore the time scale for the solution is $O(\frac{1}{\epsilon\vee\beta})=O(\frac{1}{\epsilon})$. There exists however some asymptotics models assuming small amplitude surface variation ($\epsilon = O(\mu)$) and large bottom variation $\beta = O(1)$. This is the case of the well-known Boussinesq-Peregrine model (\cite{boussinesq1},\cite{boussinesq2},\cite{peregrine}) that has been used a lot in applications. For such a regime, Theorem \eqref{david} provides an existence time of order $O(1)$ only. Our aim is here to improve this result in order to reach an $O(\frac{1}{\epsilon})$ existence time. We prove the following result.

\begin{theorem}
The Water Waves equation with surface tension \eqref{ww_equation3} admits a solution on a time interval of the form $[0;\frac{1}{\epsilon}]$ where $T$ only depends on $B_0\mu$ and on the initial data. \label{theorem_flou}
\end{theorem}

In order to prove Theorem \eqref{theorem_flou}, we use a method inspired by Bresch-Métivier \cite{bresch_metivier} and Métivier-Schochet \cite{schochetmetivier}. The only condition we need is that there is a small amount of surface tension. More precisely, we assume that the capillary parameter $\frac{1}{B_0}$ is at the same order as the shallowness parameter $\mu$.
\\

In the context of the method used by \cite{bresch_metivier}, we start by rescaling the time by setting $t'=t\epsilon$. The Theorem \ref{david} now gives an existence time of order $\epsilon$ in these new scaled variables. The Craig-Sulem-Zakharov formulation of the Water Waves problem in the newly scaled variables is 
\begin{align}\begin{cases}
       \displaystyle{ \partial_t \zeta - \frac{1}{\mu\epsilon} G\psi = 0} \\
        \ds\partial_t \psi + \frac{1}{\epsilon}\zeta + \frac{1}{2}\mid\nablag\psi\mid^2 - \frac{1}{\mu}\frac{(G\psi +\epsilon\mu\nablag\zeta\cdot\nablag\psi)^2}{2(1+\epsilon^2\mu\mid\nablag\zeta\mid^2)}=-\frac{1}{B_0\epsilon}\frac{\kappa_\gamma(\epsilon\sqrt{\mu}\zeta)}{\epsilon\sqrt{\mu}}. \label{ww_equation3}
\end{cases}\end{align}
The difficulty is therefore to handle the singular $O(\frac{1}{\epsilon})$ terms of this equation. For the sake of clarity, let us sketch the method of \cite{bresch_metivier}, \cite{schochetmetivier} on the example of the Shallow-Water equations, implemented in \cite{bresch_metivier}. The Shallow-Water equations can be read in the present setting variables  

\begin{align}
\begin{cases}
\ds \partial_t\zeta +\frac{1}{\epsilon}\nablag\cdot (h\Vu) = 0 \\
\ds \partial_t\Vu+(\Vu\cdot\nablag)\Vu + \frac{1}{\epsilon}\nablag\zeta = 0.\label{shallowater1}
\end{cases}
\end{align}
We denoted $h$ the total height of water: $$h(t,X)=1+\epsilon\zeta(t,X)-\beta b(X)$$ and $$\Vu(t,X) = \frac{1}{h(t,X)}\int_{-1+\beta b(X)}^{\epsilon\zeta(t,X)} V(t,X,z)dz$$ the vertical mean of $V$, the horizontal component of the velocity. The natural energy associated to this equation is $$E(\zeta,\Vu) = \frac{1}{2}\vert\zeta\vert_2^2+\frac{1}{2}(h\Vu,\Vu)_2.$$ By derivating in time it is easy to check that $E\big((\epsilon\partial_t)^k\zeta, (\epsilon\partial_t)^k\Vu\big)$ is uniformly bounded with respect to $\epsilon$. It is however not the case for $H^s$ norms of these unknowns with $s\geq 1$, because of commutators terms with the bottom parametrization that are of order $\frac{\beta}{\epsilon}$ and therefore singular if $\beta =O(1)$. Now, to recover an energy uniformly bounded with respect to $\epsilon$ for the spatial derivatives, we use the equation \eqref{shallowater1} to write $$\nablag\zeta = \epsilon\partial_t\Vu +\epsilon R,$$ where $\vert R\vert_2\leq  C$ with $C$ independent of $\epsilon$. Thus $\nablag\zeta$ is bounded in $L^2$ norm uniformly with respect to $\epsilon$. It allows us to recover a control of $\vert\zeta\vert_{H^1}$.  For $\Vu$, the equation \eqref{shallowater1} gives that $\vert\nablag\cdot(h\Vu)\vert_2$ is bounded by $\vert(\epsilon\partial_t)\zeta\vert_2$ (bounded uniformly), and taking the rotational of the second equation, one has that $\mbox{rot}(\Vu)$ satisfies a symmetric hyperbolic equation of the form $$\partial_t \mbox{\rm curl}\Vu+\Vu\cdot\nablag\mbox{\rm curl}\Vu  = R$$ with $$\vert R\vert_{L^2} \leq C,$$ and with $C$ independent of $\epsilon$, and thus $\mbox{\rm curl}\Vu$ is uniformly bounded in $L^2$ norm. With an ellipticity argument, one recovers an uniform bound for $\Vu$ in $H^1$ norm. By induction, one can recover the same bound for higher order space derivatives. \\

  We propose here an adaptation of this method to the Water Waves problem.  The structure of the equation is important for this method. For example, for the Water Waves equation \eqref{ww_equation3} without surface tension ($\ds \frac{1}{B_0} = 0$), one can differentiate the energy $$E(\zeta,\psi) = \frac{1}{2}(G\psi,\psi)_2+\frac{1}{2}\vert\zeta\vert_2^2$$ with respect to time, and check that \begin{equation}E\big((\epsilon\partial_t)\zeta,(\epsilon\partial_t)\psi\big)\leq C,\label{doublestar}\end{equation} where $C$ does not depend on $\epsilon$. One has, for this energy, the equivalence \begin{equation}E(\zeta,\psi)\sim \frac{1}{2}\vert\B(\epsilon\partial_t)\psi\vert^2_2+\frac{1}{2}\vert(\epsilon\partial_t)\zeta\vert^2_2\label{equivaenergy}\end{equation} where $\B$ is equivalent to the square root of the Dirichlet-Neumann operator, and acts as an order $1/2$ operator (see later Section \ref{statement}): $$\B = \frac{\D}{(1+\sqrt{\mu}\D)^{1/2}}.$$ The problem is then to recover higher order space derivatives estimates. The second equation of \eqref{ww_equation3} gives $$\zeta = \epsilon\partial_t\psi+\epsilon R$$ where $\vert\B R\vert_2\leq C$ with $C$ independent of $\epsilon$. Using \eqref{equivaenergy}, this equation only provides a bound uniform with respect to $\epsilon$ for $\vert\B\zeta\vert_2$, and one recover only $1/2$ space derivative. The same goes for the first equation: $\vert G\psi\vert_2$ is bounded uniformly with respect to $\epsilon$ by $\vert \epsilon\partial_t\zeta\vert_2$, and the ellipticity of $G$ allows us to recover a bound for $\psi$ in $H^1$ norm, and thus we only gained $1/2$ derivative with respect to the control provided by \eqref{doublestar}. We can overcome this problem by taking into account the surface tension term. Indeed, the energy for the equation \eqref{ww_equation3} is $$E(\zeta,\psi)=\vert\zeta\vert_2^2+\frac{1}{B_0}\vert \nablag \zeta\vert_2^2+\vert\B\psi\vert_2^2.$$ As explained above, it is easy to get a uniform bound of the energy for times derivatives. The first equation yields $$\frac{1}{\mu}G\psi = (\epsilon\partial_t)\zeta$$ and thus $\frac{1}{\mu}G\psi$ is bounded uniformly in $L^2$ and $\frac{1}{\sqrt{B_0}} H^1$ norm. One can then prove a uniform bound for $\vert\B\nabla\psi\vert_2$, depending on $B_0\mu$, using the ellipticity of $G$. Note that we pay special attention to the dependence of the surface tension on the existence time; it is indeed important to get a dependence on $Bo\mu$ and not only on $B_0$ since this weaken our assumption on the size of the capillary effects. 

\subsection{Notations}\label{notations}
We introduce here all the notations used in this paper.
 \subsubsection{Operators and quantities} Because of the use of dimensionless variables (see before the "dimensionless equations" paragraph), we use the following twisted partial operators: 
\begin{align*}
&\nabla^{\gamma} = {}^t(\partial_x,\gamma\partial_y) \quad &\text{ if } d=2\quad &\text{ and } &\nabla^{\gamma} = \partial_x &\quad \text{ if } d=1 \\
&\Delta^{\mu,\gamma} = \mu\partial_x^2+\gamma^2\mu\partial_y^2+\partial_z^2 \quad &\text{ if } d=2\quad &\text{ and } &\Delta^{\mu,\gamma} = \mu\partial_x^2+\partial_z^2 &\quad \text{ if } d=1 \\
&\nabla^{\mu,\gamma} = {}^t(\sqrt{\mu}\partial_x,\gamma\sqrt{\mu}\partial_y,\partial_z)\quad &\text{ if } d=2\quad &\text{ and } &{}^t(\sqrt{\mu}\partial_x,\partial_z) &\quad \text{ if } d=1 \\
&\nabla^{\mu,\gamma}\cdot = \sqrt{\mu}\partial_x+\gamma\sqrt{\mu}\partial_y+\partial_z\quad &\text{ if } d=2\quad &\text{ and } &\sqrt{\mu}\partial_x+\partial_z &\quad \text{ if } d=1 \\
&\mbox{\rm curl}^{\mu,\gamma} = {}^t(\sqrt{\mu}\gamma\partial_y-\partial_z,\partial_z-\sqrt{\mu}\partial_x,\partial_x-\gamma\partial_y)&\text {if } d=2\quad
\end{align*}
\begin{remark}All the results proved in this paper do not need the assumption that the typical wave lengths are the same in both direction, ie $\gamma = 1$. However, if one is not interested in the dependance of $\gamma$, it is possible to take $\gamma = 1$ in all the following proofs. A typical situation where $\gamma\neq 1$ is for weakly transverse waves for which $\gamma=\sqrt{\mu}$; this leads to weakly transverse Boussinesq systems and the Kadomtsev–Petviashvili equation (see \cite{lannes_saut}).
\end{remark}
 We use the classical Fourier multiplier 
$$\Lambda^s = (1-\Delta)^{s/2} \text{ on } \mathbb{R}^d$$ defined by its Fourier transform as $$\mathcal{F}(\Lambda^s u)(\xi) = (1+\vert\xi\vert^2)^{s/2}(\mathcal{F}u)(\xi)$$ for all $u\in\mathcal{S}'(\mathbb{R}^d)$.
The operator $\mathfrak{P}$ is defined as 
\begin{equation}
\mathfrak{P} = \frac{\vert D^{\gamma}\vert}{(1+\sqrt{\mu}\vert D^{\gamma}\vert)^{1/2}}\label{defp}
\end{equation}
where $$\mathcal{F}(f(D)u)(\xi) = f(\xi)\mathcal{F}(u)(\xi)$$ is defined for any smooth function $f$ and $u\in\mathcal{S}'(\mathbb{R}^d)$. The pseudo-differential operator $\mathfrak{P}$ acts as the square root of the Dirichlet Neumann operator (see later \ref{equivanorme}).
\\

We denote as before by $G_{\mu,\gamma}$ the Dirichlet-Neumann operator, which is defined as followed in the scaled variables:

\begin{equation*}
G_{\mu,\gamma}\psi = G_{\mu,\gamma}[\epsilon\zeta,\beta b]\psi = \sqrt{1+\epsilon^2 \modd{\nablag\zeta}} \partial_n \Phi_{\vert z=\epsilon\zeta} =  (\partial_z\Phi-\mu\nabla^{\gamma}(\epsilon\zeta)\cdot\nabla^{\gamma}\Phi)_{\vert z=\epsilon\zeta}.
\end{equation*} 
where $\Phi$ solves the Laplace equation 

\begin{align*}
\begin{cases}
\Delta^{\gamma,\mu}\Phi = 0\\
\Phi_{\vert z=\epsilon\zeta} = \psi,\quad \partial_n \Phi_{\vert z=-1+\beta b} = 0
\end{cases}
\end{align*}

For the sake of simplicity, we use the notation $G[\epsilon\zeta,\beta b]\psi$ or even $G\psi$ when no ambiguity is possible. \\

\subsubsection{The Dirichlet-Neumann problem}
In order to study the Dirichlet-Neumann problem \eqref{dirichlet}, we need to map $\Omega_t$ into a fixed domain (and not on a moving subset). For this purpose, we introduce the following fixed strip:
$$\mathcal{S} = \mathbb{R}^d\times (-1;0)$$
and the diffeomorphism 
\begin{displaymath}
\Sigma_t^{\epsilon} :
\left.
  \begin{array}{rcl}
    \mathcal{S} & \rightarrow &\Omega_t \\
    (X,z) & \mapsto & (1+\epsilon\zeta(X)-\beta b(X))z+\epsilon\zeta(X) \\
  \end{array}
\right.
\end{displaymath}
It is quite easy to check that $\Phi$ is the variational solution of \eqref{dirichlet} if and only if $\phi = \Phi\circ\Sigma_t^{\epsilon}$ is the variational solution of the following problem:
\begin{align}\begin{cases}
        \nabla^{\mu,\gamma}\cdot P(\Sigma_t^{\epsilon})\nabla^{\mu,\gamma} \phi = 0  \label{dirichletneumann}\\
        \phi_{z=0}=\psi,\quad \partial_n\phi_{z=-1} = 0,  \end{cases}
\end{align}
and where $$P(\Sigma_t^{\epsilon}) = \vert \det  J_{\Sigma_t^{\epsilon}}\vert J_{\Sigma_t^{\epsilon}}^{-1}~^t(J_{\Sigma_t^{\epsilon}}^{-1}),$$ where $J_{\Sigma_t^{\epsilon}}$ is the jacobian matrix of the diffeomorphism $\Sigma_t^{\epsilon}$.
For a complete statement of the result, and a proof of existence and uniqueness of solutions to these problems, see \cite{david} Chapter 2.\\

We introduce here the notations for the shape derivatives of the Dirichlet-Neumann operator. More precisely, we define the  open set  $\bold{\Gamma}\subset H^{t_0+1}(\mathbb{R}^d)^2$ as:\\
$$\bold{\Gamma} =\lbrace \Gamma=(\zeta,b)\in H^{t_0+1}(\mathbb{R}^d)^2,\quad \exists h_0>0,\forall X\in\mathbb{R}^d, \epsilon\zeta(X) +1-\beta b(X) \geq h_0\rbrace$$ and, given a $\psi\in \overset{.}H{}^{s+1/2}(\mathbb{R}^d)$, the mapping: \begin{displaymath}G[\epsilon\cdot,\beta\cdot] : \left. \begin{array}{rcl}
&\bold{\Gamma} &\longrightarrow H^{s-1/2}(\mathbb{R}^d) \\
&\Gamma=(\zeta,b) &\longmapsto G[\epsilon\zeta,\beta b]\psi.
\end{array}\right.\end{displaymath} We can prove the differentiability of this mapping. See Appendix \ref{appendixA} for more details. We denote $d^jG(h,k)\psi$ the $j$-th derivative of the mapping at $(\zeta,b)$ in the direction $(h,k)$. When we only differentiate in one direction, and no ambiguity is possible, we simply denote $d^jG(h)\psi$ or $d^j G(k)\psi$.

\subsubsection{Functional spaces}
The standard scalar product on $L^2(\mathbb{R}^d)$ is denoted by $(\quad,\quad)_2$ and the associated norm $\vert\cdot\vert_2$. We will denote the norm of the Sobolev spaces $H^s(\mathbb{R}^d)$ by $\vert \cdot\vert_{H^s}$.\\
\\We introduce the following functional Sobolev-type spaces, or Beppo-Levi spaces: 
\begin{definition}
We denote $\dot{H}^{s+1}(\mathbb{R}^d)$ the topological vector space 
$$\dot{H}^{s+1}(\mathbb{R}^d) = \lbrace u\in L^2_{loc}(\mathbb{R}^d),\quad \nabla u\in H^s(\mathbb{R}^d)\rbrace$$
endowed with the (semi) norm $\vert u\vert_{\dot{H}^{s+1}(\mathbb{R}^d)} = \vert\nabla u\vert_{H^s(\mathbb{R}^d)} $.
\end{definition}
Just remark that $\dot{H}^{s+1}(\mathbb{R}^d)/\mathbb{R}^d$ is a Banach space (see for instance \cite{lions}).
\\

The space variables $z\in\mathbb{R}$ and $X\in\mathbb{R}^d$ play different roles in the equations since the Euler formulation (\ref{euler}) is posed for $(X,z)\in \Omega_t$. Therefore, $X$ lives in the whole space $\mathbb{R}^d$ (which allows to take fractionary Sobolev type norms in space), while $z$ is actually bounded.  For this reason, we need to introduce the following Banach spaces: 
\begin{definition}
The Banach space $(H^{s,k}((-1,0) \times\mathbb{R}^d),\vert .\vert_{H^{s,k}})$ is defined by 
$$H^{s,k}((-1,0) \times\mathbb{R}^d) = \bigcap_{j=0}^{k} H^j((-1,0);H^{s-j}(\mathbb{R}^d)),\quad \vert u\vert_{H^{s,k}} = \sum_{j=0}^k \vert \Lambda^{s-j}\partial_z^j u\vert_2.$$
\end{definition}

\section{Main result}
This section is dedicated to the proof of Theorem \ref{theorem_flou}. In \S \ref{energysection} we introduce the energy space $\mathcal{E}^N_\sigma$ used in the Water Waves equations. This energy plays an important role in the proof of the main result, since the key point consists in proving that this energy is uniformly bounded with respect to $\epsilon$. We also recall in this Subsection the method used to prove the local existence theorem  for the Water Waves equation. The proof of the local existence relies on the important assumption that the Rayleigh-Taylor condition holds; this is discussed in \S \ref{rayleigh_section}. The following \S \ref{statement} states the main result of this paper, that is, the precise statement of Theorem \ref{theorem_flou}. The last \S \ref{proof_section} is dedicated to the proof of this result.
\subsection{The energy space}\label{energysection}
The purpose of this section is to introduce the energy space used in the proof of the local existence result for the Water Waves equation. To this purpose, we explain the strategy of this proof. We adapt here the approach of \cite{david} to the rescaled in time equations \eqref{ww_equation2}, pointing out where the singular terms are. We recall that we rescale the time variable for the equation \eqref{ww_equation1} by setting $$t'=t\epsilon.$$ The Water Waves equation with surface tension \eqref{ww_equation1} in the newly scaled variables is  
\begin{align}\begin{cases}
       \displaystyle{ \partial_t \zeta - \frac{1}{\mu\epsilon} G\psi = 0} \\
        \ds\partial_t \psi + \frac{1}{\epsilon}\zeta + \frac{1}{2}\mid\nablag\psi\mid^2 - \frac{1}{\mu}\frac{(G\psi +\epsilon\mu\nablag\zeta\cdot\nablag\psi)^2}{2(1+\epsilon^2\mu\mid\nablag\zeta\mid^2)}=-\frac{1}{B_0\epsilon}\frac{\kappa_\gamma(\epsilon\sqrt{\mu}\zeta)}{\epsilon\sqrt{\mu}}. \label{ww_equation2}
\end{cases}\end{align}
\begin{remark} We recall that $$\kappa_\gamma(\zeta) = -\nablag\cdot(\frac{\nablag\zeta}{\sqrt{1+\vert\nablag\zeta\vert^2}}),$$ so the surface tension term that appears on the right hand side of the second equation of \eqref{ww_equation2} is only of size $\ds \frac{1}{\epsilon}$ (thought at first sight it seems of size $\ds \frac{1}{\epsilon^2}$).\end{remark} 

The purpose of the proof of an existence time uniform with respect to $\epsilon$ for this equation in the newly time scaled variables, is to get an  uniform bound with respect to $\epsilon$ for a good quantity called the energy which controls Sobolev norms of the unknowns. By a continuity argument, one can deduce a time existence independent of $\epsilon$. For the Water Waves equation with surface tension \eqref{ww_equation2}, a natural quantity appears to act as an "energy". If we look at the linearized equation around the rest state $\zeta=0,\psi=0$, we find a system of evolution equations $$\partial_t U+\frac{1}{\epsilon}\mathcal{A}_\sigma U = 0,\quad \text{ with }\quad \mathcal{A}_\sigma=\begin{pmatrix}
&0 &-\frac{1}{\mu}G[0,\beta b] \\
&1-\frac{1}{B_0}\Delta^\gamma &0\end{pmatrix}.$$ This system can be made symmetric if we multiply it by the symmetrizer 

$$\begin{pmatrix}
&1-\frac{1}{B_0}\Delta^\gamma & 0\\
&0 & \frac{1}{\mu}G[0,\beta b]
\end{pmatrix},$$
where $U={}^t(\zeta,\psi)$. In \cite{alvarez},\cite{david}, $G[0,\beta b]$ is replaced by $G[0,0]$ in $\mathcal{A}_\sigma$. Here, we cannot perform this simplification because the error would be of size $O(\frac{\beta}{\epsilon})$ and therefore singular since $\beta = O(1)$. This suggests a natural energy of the form $$\vert\zeta\vert_2^2+\frac{1}{B_0}\vert\nablag\zeta\vert_2^2+(\frac{1}{\mu}G[0,\beta b]\psi,\psi)_2.$$ The last term is uniformly equivalent to $\vert\B\psi\vert_2^2$ where $\B$ is defined in \eqref{defp}. See later Remark \ref{energy_size} for a precise statement. Thus, $\B$ acts as the square root of the Dirichlet-Neumann operator and is of order $1/2$.
\\ 

This energy has not the sufficient order of derivatives to have a real control of the unknowns. For instance, the product $\vert\nablag\psi\vert^2$ in the second equation of $\eqref{ww_equation2}$ is not defined if $\psi$ is only $H^{1/2}(\R^d)$. To recover a control of the unknowns at a higher order, the classical scheme for this kind of method, is to differentiate the equation \eqref{ww_equation2} in order to get an evolution equation of the unknowns $\partial_{X_i}^k\zeta,\partial_{X_i}^k\psi$. For this purpose, we need to use an explicit shape derivative formula with respect to the surface for the Dirichlet Neumann operator. It is given in Appendix Theorem \ref{321}:
$$dG(h)\psi = -\epsilon G(h\underline{w})-\epsilon\mu\nablag\cdot(h\Vd)$$ with $$\underline{w} = \frac{G\psi+\epsilon\mu\nablag\zeta\cdot\nablag\psi}{1+\epsilon^2\mu\modd{\nablag\zeta}} \quad\text{ and }\quad \Vd = \nablag\psi-\epsilon\underline{w}\nablag\zeta.$$ See \cite{david} Chapter 3 for a full proof of this formula. By differentiating $N$ times the first equation of \eqref{ww_equation2}, one finds after some computation that, for $\vert(\alpha,k)\vert =N$, $$\partial_t\partial^{(\alpha,k)}\zeta +\nabla^\gamma\cdot (\underline{V}\partial^{(\alpha,k)}\zeta)-(\frac{1}{\mu\epsilon}G(\partial^{(\alpha,k)}\psi-\epsilon\underline{w}\partial^{(\alpha,k)} \zeta)+\frac{1}{\epsilon}\underset{j\in\mathbb{N^*},l_1+...+l_j+\delta = (\alpha,k)}{\sum} d^j G(\partial^{l_1}b,...,\partial^{l_j}b) \partial^{\delta}\psi )) = R,$$ with, for any $d+1$-uplet of integers $(\alpha,k) = (\alpha_1,...,\alpha_{d},k)$, $\vert(\alpha,k)\vert = \sum_{i=1}^{d} \alpha_i+k$, and with the notation $$\forall (\alpha,k)=(\alpha_1,...,\alpha_{d},k)\in\mathbb{N}^{d+1},\quad \partial^{(\alpha,k)} f = \partial_{X_1}^{\alpha_1}...\partial_{X_d}^{\alpha_d}(\epsilon\partial_t)^kf$$ and where, without entering technical details $$\vert R\vert_{H^{N-1}} \leq C$$ where $C$ does not depend on $\epsilon$. This evolution equation can be found by differentiating $N$ times the Dirichlet-Neumann operator $G[\epsilon\zeta,\beta b]\psi$ and thus $R$ contains derivatives of the form $$dG(\partial^{l_1}\zeta,...,\partial^{l_j}\zeta,\partial^{m_1}\beta b,...,\partial^{m_k}\beta b)\partial^\delta\psi$$ where the shape derivatives of $G$ with respect to $\zeta$ have an $\epsilon$ factor that cancel the $\ds \frac{1}{\epsilon}$ singularity (see also the definition of $dG$ in Section \ref{notations}). Finally, the only singular terms comes from the derivatives of the Dirichlet Neumann operator with respect to the bottom. Up to some extra singular source terms coming from the shape derivatives with respect to the bottom, this equation has the same structure as \eqref{ww_equation2}, with $\partial^{(\alpha,k)} \zeta$ and $\partial^{(\alpha,k)}\psi-\epsilon\underline{w}\partial^{(\alpha,k)} \zeta$ playing the role of $\zeta$ and $\psi$ respectively,  These quantities are the so called "Alinhac good unknowns" and will be denoted by:\begin{equation}
\forall \vert(\alpha,k)\vert\geq 1\quad \zeta_{(\alpha,k)} = \partial^{(\alpha,k)} \zeta, \quad \psi_{(\alpha,k)} = \partial^{(\alpha,k)} \psi - \epsilon \wsurf \partial^{(\alpha,k)}\zeta.\quad  \label{defpsia}\end{equation} An evolution equation for the unknown $\psi_{(\alpha,k)}$ in term of the good unknowns can be also obtained  (see later Section \ref{proof_section}). \\

In the surface tension case, the leading order operator is the surface tension term $\kappa_\gamma$. This leads to some technical complications. For instance, in order to control the time derivatives of $\kappa_\gamma$, one has to include the time derivatives of the unknowns in the energy. This method has been used by \cite{rousset}, \cite{david2} to study the Water Waves Problem with surface tension. Time derivatives and space derivatives play a different role in this proof, and we use the notation  $$\forall k\in\mathbb{N},\qquad \zetak = (\epsilon\partial_t)^k \zeta,\quad\text{ and } \psik = (\epsilon\partial_t)^k\psi-\epsilon\underline{w}(\epsilon\partial_t)^k\zeta$$ for time derivatives, and $$\forall\alpha\in\mathbb{N}^d,\forall k\in\mathbb{N},\qquad \zeta_{(\alpha,k)} = (\epsilon\partial_t)^k\partial^\alpha \zeta\quad\text{ and }\quad \psi_{(\alpha,k)} = (\epsilon\partial_t)^k\partial^\alpha \psi-\epsilon\underline{w}(\epsilon\partial_t)^k\partial^\alpha\zeta$$ such that $f_{(\alpha,k)}$ denotes indeed the good unknown defined by \eqref{defpsia} with index the $d+1$-uplet $(\alpha,k)$. See \cite{david} Chapter 9 for more details about how to handle the time derivative in the energy.\\

All these considerations explain why we do not use, for the local existence result of the Water Waves equations, an energy involving terms of the form $\partial_{X_i}^k\psi$ but rather the following energy:

\begin{equation}\mathcal{E}_\sigma^N(U) = \modd{\zeta}_2+\modd{\mathfrak{P}\psi}_{H^{t_0+3/2}}+\underset{(\alpha,k)\in\mathbb{N}^{d+1},1\leq\mid(\alpha,k)\mid\leq N}{\sum} \modd{\zeta_{(\alpha,k)}}_{H^1_\sigma}+\modd{\mathfrak{P} \psi_{(\alpha,k)}}_2\label{energie_theoreme}\end{equation} where \begin{equation}\vert f\vert_{H^1_\sigma}^2 = \vert f\vert_2^2+\frac{1}{B_0}\vert \nablag f\vert_2^2. \label{defh1}\end{equation}

The choice of $N$ is of course purely technical, and made in particular to have the different products of functions well-defined in the Sobolev Spaces used. Again, it is very important to note that the time derivatives of order less than $N$ appear in the equation.
\\

We consider solutions $U=(\zeta,\psi)$ of the Water Waves equations in the following space:

\begin{equation*}
E_{\sigma,T}^N = \lbrace U\in C(\left[ 0,T\right];H^{t_0+2}\times\overset{.}H{}^2(\mathbb{R}^d)), \mathcal{E}_\sigma^N(U(.))\in L^{\infty}(\left[ 0,T\right])\rbrace
\end{equation*}

\subsection{The Rayleigh-Taylor condition}\label{rayleigh_section}
We explained in Subsection \ref{energysection} that the Water Waves equations \eqref{ww_equation2} can be "quasilinearized". In these quasilinearized equations, a quantity appears to play an important role. It is called the "Rayleigh-Taylor coefficient" (see \cite{david} Chapter 4 and also \cite{taylor} for more details) and is defined by\begin{equation}\underline{\mathfrak{a}}(\zeta,\psi) = 1+\epsilon(\epsilon\partial_t+\epsilon \underline{V}\cdot\nabla^{\gamma})\underline{w} = -\epsilon\frac{P_0}{\rho a g}(\partial_z P)_{\vert z = \epsilon\zeta}\label{rtdef}\end{equation} where  $\underline{w} = (\partial_z\Phi)_{\vert z=\epsilon\zeta}$ and $\Vd = (\nablag\Phi)_{\vert z=\epsilon\zeta}$ are respectively  the horizontal and vertical component of the velocity $U=\nabla_{X,z}\Phi$ evaluated at the surface.  


%

The condition for strict hyperbolicity of the Water Waves system appears to be the following "Rayleigh-Taylor condition": $\underline{\mathfrak{a}}>0$. This makes the link with the classical Rayleigh-Taylor criterion  $$\underset{\mathbb{R}^d}{\inf}(-\partial_z P)_{\vert z=\epsilon\zeta}>0$$ where $P$ is the dimensionless pressure. See \cite{raileigh} for more details. Ebin (\cite{ebin}) showed that the Water Waves problem (without surface tension) is ill-posed if the Rayleigh-Taylor condition is not satisfied. Wu \cite{wu1} proved that this condition is satisfied by any solution of the Water Waves problem in infinite depth. It is proved also in \cite{david} Chapter 4 that this is also true in finite depth for the case of flat bottom. For the surface tension case, the Rayleigh-Taylor condition does not need to be satisfied in order to have well-posedness, but the existence time depends too strongly on the surface tension coefficient and is then too small for most applications to oceanography.

\subsection{Statement of the result}\label{statement}
We now state the main result.

\begin{theorem}\label{uniform_result}
Let $t_0>d/2$,$N\geq t_0+t_0\vee 2+3/2$. Let $U^0 = (\zeta^0,\psi^0)\in E_0^N,b\in H^{N+1\vee t_0+1}(\mathbb{R}^d)$. Let $\epsilon,\gamma,\beta$ be such that $$0\leq \epsilon,\beta,\gamma\leq 1,$$ and moreover assume that: \begin{equation}\exists h_{min}>0,\exists a_0>0,\qquad 1+\epsilon\zeta^0-\beta b\geq h_{min}\quad\text{ and } \quad  \rt(U^0)\geq a_0\label{rayleigh}\end{equation} Then, there exists $T>0$ and a unique solution $U^\epsilon\in E_{\sigma,T}^N$ to \eqref{ww_equation2} with initial data $U^0$. Moreover, $$\frac{1}{T}= C_1,\quad\text{ and }\quad \underset{t\in [0;T]}{\sup} \mathcal{E}_\sigma^N(U^\epsilon(t)) = C_2$$ with $\ds C_i=C(\mathcal{E}_\sigma^N(U^0),\frac{1}{h_{min}},\frac{1}{a_0},\vert b\vert_{H^{N+1\vee t_0+1}},\mu B_0)$ for $i=1,2$.
\end{theorem}

It is very important to note that $T$ does not depend on $\epsilon$ for small values of $\epsilon$. The Theorem \ref{theorem_flou} gives an existence time of order $\epsilon$ as $\epsilon$ goes to zero, if $\beta$ is of order $1$. We prove here that it is in fact, of order $1$. Note that the topography parameter $\beta$ is fixed in all this study.\\

 Let us now give a result for the Water Waves equation with surface tension \eqref{ww_equation1} in the initial time variable. The local existence Theorem \ref{david} provides an existence time of order $\frac{1}{\epsilon\vee\beta}$ to the initial Water Waves equation \eqref{ww_equation1} with surface tension. After the rescaling in time $t'=t\epsilon$, the Theorem \ref{david} provides an existence time of order $\frac{\epsilon}{\epsilon\vee\beta}\sim \frac{\epsilon}{\beta}$ as $\epsilon$ goes to $0$. Now, the main result Theorem \ref{uniform_result} claims that the existence time is in fact of order $1$ in this variable. It gives then the following result:

\begin{theorem}
Under the assumptions of Theorem \ref{uniform_result}, there exists a unique solution $(\zeta,\psi)$ of the Water Waves equation \eqref{ww_equation3} with initial condition $(\zeta^0,\psi^0)$ on a time interval $[0;\frac{ T}{\epsilon}]$ where $$\frac{1}{T}= C_1,\quad\text{ and }\quad \underset{t\in [0;T]}{\sup} \mathcal{E}_\sigma^N(U^\epsilon(t)) = C_2$$ with $\ds C_i=C(\mathcal{E}_\sigma^N(U^0),\frac{1}{h_{min}},\frac{1}{a_0},\vert b\vert_{H^{N+1\vee t_0+1}},\mu B_0)$ for $i=1,2$.
\end{theorem}
This last result gives a gain of an order $\ds \frac{1}{\epsilon}$ with respect to the time existence provided by Theorem \ref{david}. It is very important to note that the existence time given by Theorem \ref{uniform_result} depends on the constant $B_0\mu$. It implies for instance that in the shallow water limit ($\mu \ll 0$), less on less capillary effects are required (recall that the capillary effects are of order $\frac{1}{B_0}$). This is the reason why in the limit case $\mu = 0$ corresponding to the shallow water equations and investigated in \cite{bresch_metivier}, no surface tension is needed. We discuss about the shallow water regime in the Section \ref{shallow_section}.

\subsection{Proof of Theorem \ref{uniform_result}}\label{proof_section}
The key point of the proof of Theorem \ref{uniform_result} is the following Proposition: 
\begin{proposition} \label{keypoint}
Let $U^{\epsilon}=(\zeta^{\epsilon},\psi^{\epsilon})$ be the unique solution of the equation \eqref{ww_equation2} on the time interval $[0;T^\epsilon]$ and 
$$K^\epsilon = \underset{t\in [0,T^\epsilon]}{\sup} \mathcal{E}_\sigma^N(U^\epsilon(t)).$$ Then we have, with the previous notations: 
\begin{equation}
\forall t\in [0;T^{\epsilon}], \qquad\mathcal{E}_\sigma^N(U^{\epsilon})(t) \leq C_0 + C_1(K^\epsilon) (t+\epsilon)\label{desired}
\end{equation}
where $C_0=C_0(\mathcal{E}_\sigma^N(U^\epsilon_{\vert t=0}))$ and $\ds C_1(K^\epsilon) = C_1(K^\epsilon,\frac{1}{h_{min}},\frac{1}{a_0},\vert b\vert_{H^{N+1\vee t_0+1}},B_0\mu)$ are non decreasing functions of their arguments. \label{main_result}
\end{proposition}
\begin{proof}[Proof of Proposition \ref{keypoint}]

The quantity $\epsilon$ is fixed throughout the proof. We  consider  a solution $U^\epsilon = (\zeta^\epsilon,\psi^\epsilon)$ on a time interval $[0;T^\epsilon]$ of \eqref{ww_equation1} given by the standard local existence Theorem \ref{david}. To alleviate the notations, we omit, when no ambiguity is possible, the $\epsilon$ in the notation $U^{\epsilon}$ in the following estimates. Moreover, $C$ will stand for any non decreasing continuous positive function. 
Let us first sketch the proof. \\
(i) The evolution equation for time derivatives of the unknowns is "skew symmetric" with respect to $\ds \frac{1}{\epsilon}$ terms: these large terms cancel one another in energy estimates. This allows us to get the improved estimate \eqref{desired} for time derivatives:  \begin{equation}\vert (\epsilon\partial_t)^k\zeta\vert_{H^1_\sigma} + \vert \mathfrak{P}(\epsilon\partial_t)^k\psi\vert_2\leq C_1(K)t+C_0;\quad k=0..N\label{time_estimate}\end{equation} This is proved in Lemma \ref{lemma_control}. \\
(ii) To get higher order estimates (with respect to space variables), we use the equation \eqref{ww_equation2} to get $$\frac{1}{\mu}G((\epsilon\partial_t)^k\psi) = (\epsilon\partial_t)^{k+1}\zeta + \epsilon R $$ where $\vert R\vert_{H^1_\sigma} \leq C_1(K)$.  By the first step of the proof, the term $(\epsilon\partial_t)^{k+1}\zeta $ satisfies the "good" control \eqref{time_estimate} in $H^1_\sigma$ norm. Since $G$ is of order one and elliptic, this should allow us to recover one space derivative for $\B(\epsilon\partial_t)^k\psi$, with the desired control \eqref{desired}. But there is a little constraint, due to the factor $\frac{1}{B_0}$ in the definition of the $H^1_\sigma$ norm \eqref{defh1}: $$\vert f\vert_{H^1_\sigma}^2 = \vert f\vert_2^2+\frac{1}{B_0}\vert\nablag f\vert_2^2.$$ One has to use precisely the definition of the operator $\B$ given in \eqref{defp} by $$\B = \frac{\D}{(1+\sqrt{\mu}\D)^{1/2}}$$ and the inequality $$\vert \B\psi\vert_2 \leq \frac{M}{\mu}  (G\psi,\psi)_2$$ (see below Proposition \ref{equivanorme}) to get that $\D\B\psi$ is bounded in $H^1$ norm by a constant depending on $\mu B_0$. One can use the same technique to control $\Delta^\gamma\zeta$ with $(\epsilon\partial_t)^{k+1}\psi$ in $H^1_\sigma$ norm, and recover again one space derivative for $\zeta$. By finite induction, one recovers the control of the form \eqref{desired} for $\zetak$ and $\psik$. This is done in Lemma \ref{lemma_recover}. \\

For this proof, we choose (taking smaller time existence if necessary) $T^\epsilon$ such that 

\begin{equation}\forall t\in [0;T^\epsilon],\quad \rt(t) \geq \frac{a_0}{2}\quad\text{ and }\quad h(t)=1+\epsilon\zeta(t)-\beta b \geq \frac{h_{min}}{2}.\label{condition_time} \end{equation}
The first condition may be satisfied given the continuity in time of $\rt$ (see the definition of $\rt$ in \eqref{rtdef}) on the time interval $[0;T^\epsilon]$, provided that $\partial_t \rt\in L^\infty([0;T^\epsilon];\R^d)$ (this is proved in the control of $A_3$ below). The second condition is satisfied by the fact the solution $\zeta$ lives in the space $C([0;T^\epsilon];H^{t_0+2}(\R^d))$, and the continuous embedding $H^{t_0}\subset L^\infty(\R^d)$ given $t_0 >d/2$. This gives the continuity in time of $h$ (note that $b$ is also in $L^\infty(\R^d)$).\\

Now let us prove the desired estimate \eqref{desired} for time derivatives of $\psi$ and $\zeta$. Because of the energy space introduced in \eqref{energie_theoreme}, we want to control quantities like  
\begin{equation*}
\mathcal{E}_{(\alpha,k)} = \modd{\zeta_{(\alpha,k)}}_{H^1_\sigma}+\modd{\mathfrak{P} \psi_{(\alpha,k)}}_2, \qquad \vert (\alpha,k)\vert \leq N
\end{equation*}
 with $\alpha = 0$. Let $k$ be fixed. We look for the equations for the unknown $U^k=(\zetak,\psik)$. We denote the Rayleigh-Taylor coefficient by
$$\underline{\mathfrak{a}} = 1+\epsilon(\epsilon\partial_t+\epsilon \underline{V}\cdot\nabla^{\gamma})\underline{w}.$$ By differentiating $k$ times the equations  \eqref{ww_equation2} with $(\epsilon\partial_t)^k$, one find after some computations the following result:

\begin{lemma}\label{lemma_quasilinear} The unknown $\Uk$ satisfies the following equation:
\begin{equation}
\partial_t \Uk+\frac{1}{\epsilon}\mathcal{A}_\sigma[U]\Uk+B[U]\Uk+C_k[U]U_{(k-1)} = {}^t(R_k,S_k)\label{quasilinear}
\end{equation}
with the operators $$\ds \mathcal{A}_\sigma[U] = \begin{pmatrix}
&0& -\frac{1}{\mu}G \\
&\rt-\frac{1}{B_0}\nablag\cdot \K(\sqrt{\mu}\epsilon\nablag\zeta)\nablag &0
\end{pmatrix},$$
$$\ds \mathcal{B}[U] = \begin{pmatrix}
&\Vd\cdot\nablag & 0 \\
&0 &\Vd\cdot\nablag
\end{pmatrix},$$
and $$\ds \mathcal{C}_k[U] = \begin{pmatrix}
&0& -\frac{1}{\mu}dG(\epsilon\partial_t\zeta) \\
&\frac{1}{B_0\epsilon}\nablag\cdot \K_{(k)}[\sqrt{\mu}\nablag\zeta] &0
\end{pmatrix},$$
and where $$\K(\nablag\zeta) = \frac{(1+\vert\nablag\zeta\vert^2)I_d-\nablag\zeta\otimes\nablag\zeta}{(1+\vert\nablag\zeta\vert^2)^{3/2}},$$
and $$\K_{(k)}[\nablag\zeta]F = -\nablag\cdot \Big[ d\K(\nablag\partial_t\zeta)\nablag F+d\K(\nablag F)\nablag \partial_t\zeta\Big].$$
The residual ${}^t(R_k,S_k)$ satisfies the following control : \begin{equation}\vert R_k\vert_{H^1_\sigma}+\vert\mathfrak{P}S_k\vert_2\leq C_1(K).\label{reste}\end{equation}\end{lemma} 
\begin{remark}
Let us explain why the residual has to satisfy an estimate of the form \eqref{reste}. The energy for $\zetak,\psik$ is of the form $$\vert\B \psik\vert_2^2+\vert\zetak\vert_2^2+\frac{1}{B_0}\vert\nablag\zetak\vert_2^2.$$ In order to get energy estimate, we differentiate this energy with respect to time, which leads to the control of terms of the form $$\frac{1}{\mu}(\partial_t \psik,G\psik)_2,\qquad (\partial_t\zetak,\zetak)_{H^1_\sigma}.$$ In order to control these quantities, we replace $\partial_t(\zetak,\psik)$ by their expressions given by the equation \eqref{quasilinear}. All terms satisfying a control of the form \eqref{reste} are harmless for the energy estimate, since they lead to the control of terms such as $$\frac{1}{\mu} (G\psik,S_k),\qquad (\zetak,R_k)_{H^1_\sigma}$$ which is easily done.
\end{remark}
\begin{proof}[Proof of Lemma \ref{lemma_quasilinear}]
The differentiation of the first equation of $\eqref{ww_equation2}$ takes the form (recall that $G\psi$ stands for $G[\epsilon\zeta,\beta b]\psi$ and thus any derivative of $G$ with respect to $\zeta$ involves an $\epsilon$ factor)   \begin{align*}
(\epsilon\partial_t)^k\partial_t\zeta &= \frac{1}{\mu\epsilon} G(\epsilon\partial_t)^k\psi + \frac{1}{\mu\epsilon}dG((\epsilon\partial_t)^k\zeta)\psi+\frac{1}{\mu\epsilon} dG(\epsilon\partial_t\zeta)\psi_{(k-1)}\\ &+\frac{1}{\mu\epsilon}\sum_{1\leq j_1+...+j_m+l\leq k\atop 1\leq l} \epsilon^{j_1+...+j_m}dG((\epsilon\partial_t)^{j_1}\zeta,...,(\epsilon\partial_t)^{j_m}\zeta)(\epsilon\partial_t)^l\psi.\end{align*}Using the explicit shape derivative formula with respect to the surface for $G$ given by Proposition \ref{321}, we get that $$dG((\epsilon\partial_t)^k\zeta)\psi = -\epsilon G((\epsilon\partial_t)^k\zeta \underline{w})-\epsilon\mu\nablag\cdot((\epsilon\partial_t)^k\zeta\Vd),$$ and thus using the definition of $\zetak,\psik$ given by \eqref{defpsia}, one gets the following evolution equation: $$(\epsilon\partial_t)^k\partial_t\zeta +\nablag\cdot(\Vd\zetak)- \frac{1}{\mu\epsilon} G\psik-\frac{1}{\mu\epsilon} dG(\epsilon\partial_t\zeta)\psi_{(k-1)}=\frac{1}{\mu\epsilon}\sum_{1\leq j_1+...+j_m+l\leq k\atop 1\leq l} \epsilon^{j_1+...+j_m}dG((\epsilon\partial_t)^{j_1}\zeta,...,(\epsilon\partial_t)^{j_m}\zeta)(\epsilon\partial_t)^l\psi.$$The term $dG(\epsilon\partial_t\zeta)\psi_{(k-1)}$  is controlled in $L^2$ norm, but not in $H^1_\sigma$ norm. The terms of the right hand side involve derivatives of $\psi$ of order less than $N-2$ and then can be put in a residual $R_k$ with a control \eqref{reste}, using Proposition \ref{328}. The differentiation of the second equation of \eqref{ww_equation2} involves the linearization of the surface tension term $$ \frac{1}{\epsilon\sqrt{\mu}}(\epsilon\partial_t)^k \kappa_\gamma (\epsilon\sqrt{\mu}\zeta) = -\nablag\cdot\K(\epsilon\sqrt{\mu}\nablag\zeta)\nablag (\epsilon\partial_t)^k\zeta + K_{(k)}[\epsilon\sqrt{\mu}\nablag\zeta](\epsilon\partial_t)^{k-1}\zeta+...$$ The second order operator $\mathcal{K}_{(k)}$  is not controlled in $H^{1/2}$ norm (or $\vert \B\cdot\vert_2$ norm). The other terms can be put in the residual $S_k$ with a control \eqref{reste}.  See \cite{david} Chapter 9 for a complete proof of the evolution equation in terms of unknowns $\psik$,$\zetak$.\end{proof}\hfill$\qquad \Box  $
 
We now show that the singular terms of size $O(\frac{1}{\epsilon})$ are transparent in the energy estimates for the evolution equation \eqref{quasilinear}. This yields the bounds announced in \eqref{time_estimate}.

%

\begin{proposition}
One has the following estimate for all $0\leq k\leq N$: 
$$ \vert \zetak\vert_{H^1_\sigma}+\vert\mathfrak{P}\psik\vert_2\leq C_2(K)t+C_0,$$
where $C_0=C(\mathcal{E}_\sigma^N(U_{\vert t=0}))$ and $C_2=\ds C(\frac{1}{h_{\min}},\frac{1}{a_0},\vert b\vert_{H^{N+1\vee t_0+1}})$ are non decreasing function of their arguments. \label{lemma_control}
\end{proposition}
 \begin{remark}
An evolution equation can also be obtained for space derivatives, and then takes the form  \begin{equation}
\partial_t U_{(\alpha,k)}+\frac{1}{\epsilon}\tilde{\mathcal{A}}_\sigma[U]U_{(\alpha,k)}+B[U]U_{\widecheck{(\alpha,k)}}+C_{(\alpha,k)} = {}^t(R_{(\alpha,k)},S_{(\alpha,k)})
\end{equation} with $$\ds \tilde{\mathcal{A}}_\sigma[U] = \begin{pmatrix}
&0& -\frac{1}{\mu}G \\
&\rt-\frac{1}{B_0}\nablag\cdot \K(\sqrt{\mu}\epsilon\nablag\zeta)+\sum_{\vert\alpha\vert +\vert\delta\vert\leq N} dG(\partial^{\alpha_1} b,...,\partial^{\alpha_j}b)\partial^\delta \psi &0
\end{pmatrix}.$$ and 
$$U_{\widecheck{(\alpha,k)}} = \sum_{j=1}^d U_{(\alpha-e_j),k} + U_{(\alpha,k-1)}$$ with $e_j$ the unit vector in the $j-th$ direction.
This system is then non symmetrizable with respect to $\ds \frac{1}{\epsilon}$ terms, and the controls are not uniform with respect to $\epsilon$, due to spatial derivatives of the bottom. This is the reason why we have to control the time derivatives first, and then use the structure of the equation to recover higher order derivatives. In the case of a flat bottom $\beta=0$, or almost flat bottom $\beta = O(\epsilon)$, the terms involving space derivatives of $b$ in $\tilde{\mathcal{A}}_\sigma[U]$ can be put in the residual and are easy to control. The proof of Theorem \ref{david} as considered in \cite{alvarez} gives then a time existence of order $\frac{1}{\epsilon}$. 
 \end{remark}
\begin{proof}[Proof of Lemma \ref{lemma_control}]
The system \eqref{quasilinear} can be symmetrized with respect to main order terms of size $\ds \frac{1}{\epsilon}$ if we multiply it by the operator \begin{equation}
\Sy=\begin{pmatrix}
&\rt-\frac{1}{B_0}\nablag\cdot \K(\epsilon\sqrt{\mu}\nablag\zeta)\nablag&0\\
&0&\frac{1}{\mu}G
\end{pmatrix}.\label{symmetrizer1}\end{equation}
This suggests to introduce 
\begin{align}&E_0=\frac{1}{2}\vert\zeta\vert_{H^1_\sigma}^2+\frac{1}{2\mu}(G\psi,\psi)_2+\frac{1}{2B_0}(\frac{\nablag\zeta}{\sqrt{1+\epsilon^2\mu\vert\nablag\zeta\vert^2}},\nablag\zeta)_2,\qquad &k=0\nonumber \\
&E_k = (\Sy \Uk,\Uk)_2 .\qquad  &k\neq 0\label{energie_twisted}
\end{align}
The quantity $\ds\sum_{k=0}^N E_k$ is uniformly equivalent to the energy $\mathcal{E}_\sigma^N$ introduced in \eqref{energie_theoreme}:
\begin{lemma}\label{energy_size}
There exists $M=C(\frac{1}{h_{\min}},\vert\zeta\vert_{H^{t_0+1}},\vert b\vert_{H_{t_0+1}},\frac{1}{a_0},K)$  where $C$ is a non decreasing function of its arguments, such that $$\frac{1}{M}\mathcal{E}_\sigma^N \leq \ds\sum_{k=0}^N E_k\leq M\mathcal{E}_\sigma^N.$$ \end{lemma}
\begin{proof}[Proof of Lemma \ref{energy_size}] This is proved in \cite{david}; we give here the main steps of the proof for the sake of completeness. \\

(i) We start to use the following inequalities (see \cite{david} Chapter 3):

\begin{equation*}
(\psi,\frac{1}{\mu}G\psi)_2 \leq M_0\vert\mathfrak{P}\psi\vert_2^2\quad\text{ and }\quad \vert\mathfrak{P}\psi\vert_2^2 \leq M_0(\psi,\frac{1}{\mu}G\psi)_2 \label{equivanorme}
\end{equation*}
for all $\psi\in \overset{.}H{}^{1/2}(\mathbb{R}^d)$, where $M_0$ is a constant of the form $C(\frac{1}{h_{\min}},\vert\zeta\vert_{H^{t_0+1}},\vert b\vert_{H_{t_0+1}})$. The same inequality stands for space and time derivatives. 

(ii)Thanks to the Rayleigh-Taylor condition \eqref{condition_time}, we have also:
$$\frac{1}{M}\frac{1}{2}\modd{\zeta}\leq \frac{1}{2}(\zeta,\rt\zeta)_2 \leq M \frac{1}{2}\modd{\zeta},$$with $M$ a constant of the form $C(\frac{1}{a_0},K)$. The same inequality stands for space and time derivatives.

(iii) At last, $\mathcal{K}(\epsilon\sqrt{\mu}\nablag\zeta)$ is a $d\times d$ symmetric matrix uniformly bounded with respect to time and $\epsilon$, 
$$\frac{1}{M}\vert\zetak\vert_{H^1_\sigma}\leq (\mathcal{K}(\epsilon\sqrt{\mu}\nablag\zeta)\nablag\zetak,\nablag\zetak)_2\leq M\vert\zetak\vert_{H^1_\sigma}$$ where $M$ is a constant of the form $C_1(K)$. \hfill$\qquad \Box  $
\end{proof}

Because the term $\mathcal{C}_k[U]U^{(k-1)}$ which appear in the equation \eqref{quasilinear} contains order two derivatives with respect to $\zeta_{(k-1)}$, the time derivative of the energies $E_k$ are actually not controlled by the energy $\mathcal{E}^N_\sigma$. To overcome this problem, we slightly adjust the twisted energy $E_k$ for $k=N$ by defining 
\begin{align*}
F_k &= \epsilon(\sy U_{(k-1)},U_{(k)})_2 &\text{ if } k=N,\\ &=0&\text{ if }k\neq N,
\end{align*}
where \begin{equation*}
\sy = \begin{pmatrix}
&\frac{1}{B_0}\Kk &0 \\
&0 &\frac{1}{\mu}dG(\epsilon\partial_t\zeta)
\end{pmatrix}.
\end{equation*}
The presence of the $\epsilon$ in $F_k$ is a consequence of the first time scaling $t'=t\epsilon$. The matrix operator $\sy$ symmetrizes the subprincipal term $\mathcal{C}_k$. One derives with respect to time this "energy". Our goal is to have, for all $0\leq k\leq N$ \begin{equation*}
\frac{d}{dt}(E_k+F_k) \leq C_1(K).
\end{equation*}
We will at the end recover a similar estimate for the energy $E_k$ by a Young inequality in the control of $F_k$ by the $E_j$.\\

\textbf{Control of $\frac{d}{dt}E_0$}\\

One get, using the symmetry of $G$:  $$\frac{dE_0}{dt} = (\partial_t\zeta,\zeta)_2+\frac{1}{B_0}(\frac{\nablag\zeta}{\sqrt{1+\epsilon^2\mu\vert\nablag\zeta\vert^2}},\nablag\partial_t\zeta)_2+\frac{1}{\mu}(G\psi,\partial_t\psi)_2+A_1,$$ where, the commutator terms $[G,\partial_t]$ and $[\frac{1}{\sqrt{1+\epsilon^2\mu\vert\nablag\zeta\vert^2}},\partial_t]$ are $$A_1 = -\big(\frac{1}{2B_0}\frac{(\epsilon^2\mu\nablag\zeta\cdot\partial_t\nablag\zeta) \nablag\zeta}{(1+\epsilon^2\mu\vert\nablag\zeta\vert^2)^{3/2}},\nablag\zeta\big)_2 +\frac{1}{2\mu}(dG(\epsilon\partial_t\zeta)\psi,\psi)_2.$$ Using the equations  $\eqref{ww_equation2}$ to replace $\partial_t\zeta$ and $\partial_t\psi$ in this equality, one can write 

\begin{align*}\frac{dE_0}{dt} &= \frac{1}{\mu\epsilon}(G\psi,\zeta)_2-\frac{1}{\mu\epsilon}(G\psi,\zeta)_2 +\frac{1}{\mu\epsilon}\frac{1}{B_0}(\frac{\nablag\zeta}{\sqrt{1+\epsilon^2\mu\vert\nablag\zeta\vert^2}},\nablag G\psi)_2-\frac{1}{\mu\epsilon}\frac{1}{B_0}(\frac{\nablag\zeta}{\sqrt{1+\epsilon^2\mu\vert\nablag\zeta\vert^2}},\nablag G\psi)_2 
\\&+ A_1+B_1+B_2,\end{align*} where  

$$B_1 = -\frac{1}{2\mu}(\vert\nablag\psi\vert^2,G\psi)_2,$$
$$B_2 = \frac{1}{\mu}\big(\frac{(G\psi+\mu\nablag(\epsilon\zeta)\cdot\nablag\psi)^2}{2(1+\epsilon^2\mu\vert\nablag \zeta\vert^2)},\frac{1}{\mu}G\psi\big)_2.$$
The large terms of order $\ds\frac{1}{\epsilon}$ cancel one another, thanks to the symmetry of the equation. One must now control $A_1,B_1$ and $B_2$ in order to get the desired estimate for $E_0$.\\

\textit{- Control of $A_1$}
Let us start with the first term of $A_1$. We use the fact that $\vert\epsilon\partial_t\zeta\vert_{H^1_\sigma}$ and $\vert\zeta\vert_{H^1_\sigma}$ are bounded by $C_2(K)$, since $N\geq 2$. Moreover, since $N\geq t_0+1$, one has that $\nablag\zeta\in L^\infty(\mathbb{R}^d)$ with a control by $C_2(K)$, and thus one can write, using Cauchy-Schwartz inequality,
\begin{align*}\vert -\big(\frac{1}{2B_0}\frac{(\epsilon^2\mu\nablag\zeta\cdot\partial_t\nablag\zeta) \nablag\zeta}{(1+\epsilon^2\mu\vert\nablag\zeta\vert^2)^{3/2}},\nablag\zeta\big)_2\vert &\leq \frac{1}{2} \vert\epsilon\partial_t\zeta\vert_{H^1_\sigma}\vert \zeta\vert_{H^1_\sigma}\vert \nablag\zeta\vert_{L^\infty(\mathbb{R}^d)}^2 \\
&\leq C_2(K). \end{align*}

To control the second term of $A_1$, we use the Proposition \ref{328} in the Appendix with $s=0$ to write
\begin{align*}
\vert \frac{1}{2\mu}(dG(\epsilon\partial_t\zeta)\psi,\psi)_2\vert & \leq  M_0\vert \epsilon\partial_t\zeta\vert_{H^{t_0+1}}\vert\B\psi\vert_2\vert\B\psi\vert_2,
\end{align*}
where $M_0$ is a constant of the form $C(\frac{1}{h_{\min}},\vert\zeta\vert_{H^{t_0+1}},\vert b\vert_{H^{t_0+1}}).$ 
Moreover, $\epsilon\partial_t\zeta$ is controlled in $H^{t_0+1}$ norm by $C_2(K)$, since $N\geq t_0+2$ and thus, one has 
\begin{align*}
\vert \frac{1}{2\mu}(dG(\epsilon\partial_t\zeta)\psi,\psi)_2\vert & \leq  C_2(K).
\end{align*}
\textit{- Control of $B_1$}
One has to remark that $\nablag\psi$ is $L^\infty(\mathbb{R}^d)$. Indeed, one has $$\vert\Lambda^{t_0}\nablag\psi\vert_2\leq C \vert\mathfrak{P}\psi\vert_{H^{t_0+3/2}},$$ where $C$ does not depends on $\mu$ nor $\psi$. Since this last term is controlled by the energy, one has that $\nablag\psi\in L^\infty(\mathbb{R}^d)$ with a control by $C_2(K)$. Thus, one can write 
\begin{align*}
\vert\frac{1}{2\mu}(\vert\nablag\psi\vert^2,G\psi)_2\vert&\leq \vert\nablag\psi\vert_{L^\infty(\mathbb{R}^d)}\frac{1}{\mu}\vert G\psi\vert_2\vert\nablag\psi\vert_2\\
&\leq C_2(K).
\end{align*}
Now, using the second point of Proposition \ref{314} with $s=1/2$, one gets that 

\begin{align*}\vert\frac{1}{2\mu}(\vert\nablag\psi\vert^2,G\psi)_2\vert\leq C_2(K)M(3/2)\vert\B\psi\vert_{H^{1}}\end{align*}
where $M(3/2)$ is a constant of the form $\ds C(\frac{1}{h_0},\vert\zeta\vert_{H^{t_0+1}},\vert b\vert_{H^{t_0+1}},\vert\zeta\vert_{H^{3/2}},\vert b\vert_{H^{3/2}})$.
Just note that the first point of Proposition \ref{314} does not suffice here, since we want a control of $\frac{1}{\mu}G\psi$ and not only of $\frac{1}{\mu^{3/4}}G\psi$. This is the interest of the second point of this Proposition: the Dirichlet-Neumann operator $G$ has to be seen as a $3/2$ order operator in order to be controlled uniformly with respect to $\mu$. \\

\textit{- Control of $B_2$}
We already noticed in the previous controls that $\nablag\zeta$ and $\nablag\psi$ were in $L^\infty(\mathbb{R}^d)$ with a $C_2(K)$ control. Moreover, by Proposition \ref{314} with $s=t_0+1/2$, one gets \begin{align*}\frac{1}{\mu}\vert \Lambda^{t_0} G\psi\vert_2 &\leq M(t_0+3/2)\vert\B\psi\vert_{H^{t_0+1/2}} \\
&\leq C_2(K)\end{align*} and thus $\frac{1}{\mu}G\psi$ is $L^\infty(\mathbb{R}^d)$ with a $C_2(K)$ control. Then, we write, using Cauchy-Schwartz inequality,

\begin{align*}
B_2&\leq\frac{1}{\mu} \vert G\psi\vert_{L^\infty(\mathbb{R}^d)}+\epsilon\mu\vert\nablag\zeta\vert_{L^\infty(\mathbb{R}^d)}\vert\nablag\psi\vert_{L^\infty(\mathbb{R}^d)})\bigg( \vert\nablag\psi\vert_{L^\infty(\mathbb{R}^d)}\epsilon\vert\nablag\zeta\vert_2+\vert G\psi\vert_2\bigg)\frac{1}{\mu}\vert G\psi\vert_2 \\
&\leq C_2(K)
\end{align*}
where we used again Proposition \ref{314} to control $\frac{1}{\mu}G\psi$ in $L^2$ norm. \\

\textit{- Synthesis} To conclude, we proved that $$\frac{dE_0}{dt}\leq C_2(K),$$ which gives, by integrating in time, the following inequality: $$\forall t\in[0;T^\epsilon],\qquad E_0(t)\leq C_2(K)t+C_0$$ where $C_0$ only depends on the norm of the initial data.

\textbf{Control of $\frac{d}{dt}(E_k+F_k)$ for $k\neq 0$} \\

We deal here with the case $k=N$ (if $k<N$, there is no term of order more than $N$ that appears in the derivative of the energy). Recall that for $k=N$, we have $$E_k+F_k = (\Sy \Uk,\Uk)_2+\epsilon(\sy U_{(k-1)},\Uk)_2.$$ Therefore, by derivating in time, and using the symmetry of $\Sy$ we get 

\begin{align*}
\frac{d}{dt}(E_k+F_k) &= (\Sy\Uk,\partial_t \Uk)_2+\epsilon(\sy U_{(k-1)},\partial_t\Uk)_2\\
&+(\big[\partial_t,\Sy\big]\Uk,\Uk)_2+(\partial_t(\sy U_{(k-1)}),\Uk)_2.
\end{align*}
We replace $\partial_t\Uk$ by its expression given in the quasilinear system \eqref{quasilinear}. One gets 

\begin{align*}
\frac{d}{dt}(E_k+F_k) &= (\Sy\Uk,-\frac{1}{\epsilon}\mathcal{A}_\sigma[U]\Uk-B[U]\Uk-\mathcal{C}_k[U]U_{(k-1)})_2 \\
& + \epsilon(\sy U_{(k-1)},-\frac{1}{\epsilon}\mathcal{A}_\sigma[U]\Uk-B[U]\Uk-\mathcal{C}_k[U]U_{(k-1)})_2 \\
&+(\big[\partial_t,\Sy\big]\Uk,\Uk)_2+(\partial_t(\sy U_{(k-1)}),\Uk)_2 \\
&+(\Sy\Uk,{}^t(S_k,R_k))_2+\epsilon(\sy U_{(k-1)},{}^t(S_k,R_k))_2.
\end{align*}
Thanks to the symmetry, the large terms of size $\ds \frac{1}{\epsilon}$ cancel one another, ie $$\frac{1}{\epsilon}(\Sy\Uk,\mathcal{A}_\sigma[U]\Uk)_2 = 0.$$ This is fundamental and based on the fact that the evolution equation for the unknown $\Uk$ is still symmetrizable with respect to $\frac{1}{\epsilon}$ terms. Again, it is not the case for the evolution equation in term of spatial derivatives $ U_{(\alpha,k)}$. Then, the commutators between $\Uk$ and subprincipal terms of $\mathcal{C}_k$, which are not controlled by the energy (mainly because of the order two operators) also cancel one another, because of the choice of $F_k$. More precisely, one gets $$-(\Sy\Uk,\mathcal{C}_k[U]U_{(k-1)})_2-(\sy U_{(k-1)},\mathcal{A}_\sigma[U]\Uk)_2=0.$$ One can also check that $$\epsilon(\sy U_{(k-1)},\mathcal{C}_k[U]U_{(k-1)})_2 = 0. $$To conclude, we proved that 

$$\frac{d}{dt}(E_k+F_k) = A_1+A_2+A_3+A_4+B_1+B_2+B_3+B_4+B_5+B_6,$$
where $(\big[\partial_t,\Sy\big]\Uk,\Uk)_2=A_1+A_2+A_3$ with

$$A_1 = \frac{1}{2\mu}(dG(\epsilon\partial_t\zeta)\psik,\psik)_2,$$ 
$$A_2 = \frac{1}{2B_0}(\partial_t\ \big(K(\epsilon\sqrt{\mu}\nablag\zeta)\big)\nablag\zetak,\nablag\zetak\bigg)_2,$$
$$A_3 =\frac{1}{2} ((\partial_t\rt)\zetak,\zetak)_2.$$
The term $(\partial_t(\sy U_{(k-1)}),\Uk)_2=A_4$ is given by
$$A_4 = \frac{1}{\mu}(\epsilon\partial_t\big(dG(\epsilon\partial_t\zeta)(\psi_{(k-1)})\big),\psik)_2+\frac{1}{B_0}(\epsilon\partial_t\bigg(\mathcal{K}_{(k)}[\sqrt{\mu}\nablag\zeta]\zeta_{(k-1)}\bigg),\zetak)_2.$$
We denote the terms coming from the evolution equation contained in $-(\Sy\Uk-B[U]\Uk)_2$  by

$$B_1 = (-\Vd\nablag\cdot\zetak,\rt\zetak)_2-\frac{1}{B_0}(\nablag (\Vd\cdot\nablag\zeta),\mathcal{K}(\sqrt{\mu}\epsilon\nablag\zeta)\nablag\zetak)_2$$ and
$$B_5 = -(\Vd\cdot\nablag\psik,\frac{1}{\mu}G\psi_k)$$
and the term $-\epsilon(\sy U_{(k-1)},B[U]\Uk)_2$ is 
$$B_2 = -\epsilon\frac{1}{\mu}(dG(\partial_t\zeta)\psi_{(k-1)},\Vd\cdot\nablag\psik)_2-\frac{1}{B_0}\epsilon(\mathcal{K}_{(k)}[\sqrt{\mu}\nablag\zeta]\zeta_{(k-1)},\Vd\cdot\nablag\zetak)_2.$$
Finally, the residual terms $(\Sy\Uk,{}^t(S_k,R_k))_2+\epsilon(\sy U_{(k-1)},{}^t(S_k,R_k))_2$ which are the most easy terms to control are denoted by 
$$B_3 = (R_k,\rt\zetak)_2,$$
$$B_4 =  +\epsilon\frac{1}{\mu}(dG(\partial_t\zeta)\psi_{(k-1)},S_k)_2+\epsilon(\mathcal{K}_{(k)}[\sqrt{\mu}\nablag\zeta]\zeta_{(k-1)},R_k)_2$$ and
$$B_6 = (S_k,\frac{1}{\mu}G\psik)_2.$$
We then need to control all these terms by a constant of the form $C_2(K)$ in order to get the desired estimate. These controls requires bounds for quantities such as $dG(h,k)\psi$ or $G\psi$, which can be obtained by the use of Propositions \ref{314}, \ref{318} and \ref{328}.

\textit{-Control of $A_1$}
Using Proposition \ref{328}  with $s=0$, one gets \begin{align*}
\vert \frac{1}{2\mu}(dG(\epsilon\partial_t\zeta)\psik,\psik)_2\vert &\leq  M_0\vert\epsilon\partial_t\zeta\vert_{H^{t_0+1}}\vert\mathfrak{P}\psik\vert_2^2\\
&\leq C_2(K)
\end{align*}
since $N\geq t_0+2$. \\

\textit{-Control of $A_2$}
We start to check  that $\partial_t(K(\epsilon\sqrt{\mu}\nablag\zeta))$ is bounded in $L^\infty(\mathbb{R}^d)$ by $C_2(K)$:

\begin{align*}\partial_t(K(\epsilon\sqrt{\mu}\nablag\zeta)) &= \frac{2\epsilon^2\mu\nablag\zeta\cdot\partial_t\nablag\zeta Id-\epsilon^2\mu(\nablag\zeta\otimes\nablag\partial_t\zeta+\nablag\partial_t\zeta\otimes\nablag\zeta)}{(1+\epsilon^2\mu\vert\nablag\vert^2)^{3/2}}\\
&-\frac{3}{2}\frac{(1+\epsilon^2\mu\vert\nablag\zeta\vert^2Id-\epsilon^2\mu\nablag\zeta\otimes\nablag\zeta)(2\epsilon^2\mu\partial_t\nablag\zeta\cdot\nablag\zeta)}{(1+\epsilon^2\mu\vert\nablag\zeta\vert^2)^{5/2}}.\end{align*}

Because of the energy expression given by \eqref{energie_theoreme}, one has that $\nablag\zeta\in L^\infty(\R^d)$ (since $N\geq t_0+1$), $\epsilon\partial_t\nablag\zeta \in L^\infty(\R^d)$ (since $N\geq t_0+2$) and thus we get  
$$\vert \partial_t(K(\epsilon\sqrt{\mu}\nablag\zeta))\vert_{L^\infty(\R^d)}\leq C_2(K).$$\\

Thus, one can write 
\begin{align*}
\left\vert \frac{1}{2B_0}\bigg(\partial_t\ \big(K(\epsilon\sqrt{\mu}\nablag\zeta)\big)\nablag\zetak,\nablag\zetak\bigg)_2\right\vert &\leq C_2(K) \frac{1}{B_0}\vert\nablag\zetak\vert^2 \\
&\leq C_2(K).
\end{align*}
\textit{-Control of $A_3$}
We prove that $$\vert\partial_t\rt\vert_{L^\infty(\R^d)}\leq C_2(K).$$ Recall that $$\underline{\mathfrak{a}} = 1+\epsilon(\epsilon\partial_t+\epsilon \underline{V}\cdot\nabla^{\gamma})\underline{w}.$$

By derivating with respect to time, one get $$\partial_t\rt = (\epsilon\partial_t)^2\underline{w}+(\epsilon\partial_t)\Vd\cdot\nablag\underline{w}+\Vd\cdot\nablag(\epsilon\partial_t)\underline{w}.$$

We need to use an explicit expression of the horizontal and vertical component of the velocity at the surface $\Vd$ and $\underline{w}$ here:

$$\underline{w} = \frac{G\psi+\epsilon\mu\nablag\zeta\cdot\nablag\psi}{1+\epsilon^2\mu\vert\nablag\zeta\vert^2},\quad\text{ and }\quad \Vd = \nablag\psi-\epsilon\underline{w}\nablag\zeta. $$ If one takes brutally the $\epsilon\partial_t$ and $(\epsilon\partial_t)^2$ derivatives of these expressions, one has to deal with terms of the form $$d^2G((\epsilon\partial_t)^2\zeta)\psi,\quad d^2G(\epsilon\partial_t\zeta,\epsilon\partial_t\zeta)\psi,\quad dG(\epsilon\partial_t\zeta)(\epsilon\partial_t\psi),\quad G(\epsilon\partial_t)^2\psi.$$

For the first and second terms, we use Proposition \ref{328b} with $s=t_0+1/2$ to get 
\begin{align*}
\vert d^2G(\epsilon\partial_t\zeta,\epsilon\partial_t\zeta)\psi\vert_{H^{t_0}} &\leq M_0\mu^{3/4}\vert (\epsilon\partial_t)^2\zeta\vert_{H^{t_0+1}}^2\vert\B\psi\vert_{H^{t_0+1/2}} \\
&\leq C_2(K)
\end{align*}
since $N\geq t_0+3$. For the third term, we apply the same result to get 

\begin{align*}
\vert dG(\epsilon\partial_t\zeta)(\epsilon\partial_t\psi)\vert_{H^{t_0}}\leq M_0\mu^{3/4}\vert\epsilon\partial_t\zeta\vert_{H^{t_0+1}}\vert\B\epsilon\partial_t\psi\vert_{H^{t_0+1/2}}.
\end{align*}
There is a bit more work to achieve in order to control the term $\vert\B\epsilon\partial_t\psi\vert_{H^{t_0+1/2}}$.  We write, noticing that $N\geq t_0+3/2,$
\begin{align*}
\vert\B\epsilon\partial_t\psi\vert_{H^{t_0+1/2}} &\leq \sum_{\beta\in\mathbb{N}^d,\vert\beta\vert\leq N-1}\vert\B\epsilon\partial_t\partial^\beta\psi\vert_2 \\
&\leq \sum_{\beta\in\mathbb{N}^d,\vert\beta\vert\leq N-1} \vert \B\psi_{(\beta,1)}+\B\epsilon\underline{w}\zeta_{(\beta,1)}\vert_2
\end{align*}
using the definition of $\psi_{(\alpha,k)}$ given by \eqref{defpsia}. Using the fact that $\vert\B f\vert_2\leq C\mu^{-1/4}\vert f\vert_{H^{1/2}}$, we get  

\begin{align*}
\vert\B\epsilon\partial_t\psi\vert_{H^{t_0+1/2}} &\leq \sum_{\beta\in\mathbb{N}^d,\vert\beta\vert\leq N-1}\vert \B\psi_{(\beta,1)}\vert_2+\mu^{-1/4}\epsilon\vert \underline{w}\zeta_{(\beta,1)}\vert_{H^{1/2}} \\
&\leq \sum_{\beta\in\mathbb{N}^d,\vert\beta\vert\leq N-1}\vert \B\psi_{(\beta,1)}\vert_2+\mu^{-1/4}\epsilon\vert \underline{w}\vert_{H^{t_0}}C_2(K) \\
\end{align*}
where we used the Sobolev estimate $\vert f g\vert_{1/2} \leq C \vert f\vert_{H^{t_0}} \vert g\vert_{H^{1/2}}$ to derive the last inequality. At last, we use Proposition \ref{314} with $\underline{w}$ (see the remark at the end of this Proposition) and with $s=t_0+1/2$ to get 
\begin{align*}
\vert\B\epsilon\partial_t\psi\vert_{H^{t_0+1/2}} &\leq \sum_{\beta\in\mathbb{N}^d,\vert\beta\vert\leq N-1}\vert \B\psi_{(\beta,1)}\vert_2+\epsilon\mu^{1/2}M(t_0+1)\vert\B\psi\vert_{H^{t_0+1}}C_2(K) \\
&\leq C_2(K)
\end{align*}
and finally we proved $$\vert dG(\epsilon\partial_t\zeta)(\epsilon\partial_t\psi)\vert_{H^{t_0}}\leq C_2(K).$$

It remains $\vert G(\epsilon\partial_t)^2\psi\vert_{H^{t_0}}$ to be controlled. We use Proposition \ref{314} to write, with $s=t_0+1/2$,
$$\vert G(\epsilon\partial_t)^2\psi\vert_{H^{t_0}} \leq \mu^{3/4} M(t_0+1)\vert\B(\epsilon\partial_t)^2\psi\vert_{H^{t_0+1}}$$
and we use the previous technique to prove that $\vert\B(\epsilon\partial_t)^2\psi\vert_{H^{t_0+1}} \leq C_2(K)$.\\

Combining all these results, one can control all term of $\partial_t\rt$ in $L^\infty$ norm by $C_2(K)$, and one get the desired result: $$\vert \partial_t\rt\vert_{L^\infty(\R^d)}\leq C_2(K).$$ 
It is now easy to get 
\begin{align*}\vert \frac{1}{2} ((\partial_t\rt)\zetak,\zetak)_2\vert &\leq C_2(K)\vert\zetak\vert_2^2\\
&\leq C_2(K).
\end{align*}

\textit{- Control of $A_4$}
The fact that we get an $\epsilon\partial_t$ derivative here, and not only a $\partial_t$ derivative is essential here. The first term involves terms of the form $$(dG(\zeta_{(j)})\psi_{(l)},\psik)_2$$ with $j,l\leq N$ and the Proposition \ref{328} with $s=0$ allows to control them by $C_2(K)$. The second term of $A_4$ involves $\frac{1}{B_0}\nablag\zetak$ and $\zetak$ terms, which are controlled in $L^2$ norm, and other $L^{\infty}$ terms (see the control of $A_2$ for example). There is no other difficulty than computation, to control $A_2$ by $C_2(K)$. \\

\textit{-Control of $B_1$}
The control of the first term requires a classical symmetry trick. We write, by integrating by parts,

\begin{align*}(\Vd\cdot\nabla^\gamma\zetaa,\rt\zetaa)_2 &= -(\zetaa,\nabla^\gamma\cdot(\rt\zetaa\Vd))_2 \\
&=-(\zetaa,\nabla^\gamma\cdot(\rt\Vd)\zetaa)_2-(\zetaa,\rt\Vd\cdot\nabla^\gamma\zetaa)_2\end{align*} and thus one gets

$$(\Vd\cdot\nabla^\gamma\zetaa,\rt\zetaa)_2 = -\frac{1}{2}(\zetaa,\nabla^\gamma\cdot(\rt\Vd)\zetaa)_2.$$ We can use the same technique as in the control of $A_3$ to get $$\vert \nablag\cdot(\rt\Vd)\vert_{L^\infty(\R^d)}\leq C_2(K)$$ and we then get that 

\begin{align*}\vert (\Vd\cdot\nabla^\gamma\zetaa,\rt\zetaa)_2\vert &\leq \vert \nablag\cdot(\rt\Vd)\vert_{L^\infty(\R^d)} \vert \zetak\vert_2^2 \\
&\leq C_2(K). \end{align*}

It is the same trick for the second term. Using the symmetry of $\K$, we write
\begin{align*}
\frac{1}{B_0}(\Vd\nablag\cdot(\nablag\zetak),\Km\nablag\zetak)_2 &= \frac{1}{B_0}(\nablag\cdot(\nablag\zetak),\Km\Vd\cdot\nablag\zetak)_2
\end{align*}
and, by integrating by parts,
\begin{align*}
\frac{1}{B_0}(\Vd\nablag\cdot(\nablag\zetak),\Km\nablag\zetak)_2 &= - \frac{1}{B_0}(\nablag\zetak,\nablag\cdot(\Km\Vd)\nablag\zetak)_2\\& - \frac{1}{B_0}(\nablag\zetak,\Km\Vd\nablag\cdot(\nablag\zetak))_2 \\
&=- \frac{1}{B_0}(\nablag\zetak,\nablag\cdot(\Km\Vd)\nablag\zetak)_2\\&-\frac{1}{B_0}(\Vd\nablag\cdot(\nablag\zetak),\Km\nablag\zetak)_2
\end{align*}
and thus 
$$\frac{1}{B_0}(\Vd\nablag\cdot(\nablag\zetak),\Km\nablag\zetak)_2 = -\frac{1}{2B_0}(\nablag\zetak,\nablag\cdot(\Km\Vd)\nablag\zetak)_2.$$
We can then use the same type of computation as for the control of $A_2$ to show that $$\nablag\cdot(\Km\Vd) \in L^\infty(\R^d)$$ with a $C_2(K)$ bound, and finally we get, by Cauchy-Schwartz's inequality 
\begin{align*}\vert \frac{1}{B_0}(\Vd\nablag\cdot(\nablag\zetak),\Km\nablag\zetak)_2\vert &\leq C_2(K)\frac{1}{B_0}\vert\nablag\zetak\vert^2 \\
&\leq C_2(K).
\end{align*}
\textit{- Control of $B_2$}
For the first term, we write $$(\Vd\cdot\nablag\psik,\frac{1}{\mu}dG(\epsilon\partial_t\zeta)\psi_{(k-1)}) =(\frac{\nablag}{(1+\sqrt{\mu}\D)^{1/2}}\psik,(1+\sqrt{\mu}\D)^{1/2}(\Vd\frac{1}{\mu}dG(\epsilon\partial_t)\psi_{(k-1)})$$ and we deduce with Cauchy-Schwartz inequality that this quantity is bounded in absolute value by $$\vert\Vd\vert_{H^{t_0}}\vert\B\psik\vert_2(\frac{1}{\mu}\vert dG(\epsilon\partial_t\zeta)\psi_{(k-1)}\vert_2+\mu^{-3/4}\vert dG(\epsilon\partial_t\zeta)\vert_{H^{1/2}}).$$ We now use Proposition \eqref{328b} to control this term by $C_2(K)$. The second term of $B_2$ is controlled by using Cauchy-Schwartz inequality. \\

\textit{-Control of $B_3$}
We use the control $\vert\rt\vert_{L^\infty(\R^d)}\leq C_2(K)$ and the control over $R_k$ given by \eqref{reste} to get, with Cauchy-Schwartz inequality \begin{align*}\vert( R_k,\rt\zetak)_2\vert &\leq \vert\rt\vert_{L^\infty(\R^d)} \vert R_k\vert_2\vert\zetak\vert_2 \\
&\leq C_2(K).
\end{align*}

\textit{ - Control of $B_4$}
It is a direct use of Cauchy-Schwartz inequality and Proposition \ref{328}, and the control over $R_k$ and $S_k$ given by \eqref{reste}. \\

\textit{- Control of $B_5$}
To control this term, we use a direct application of Proposition \ref{329} to write 

\begin{align*}
\vert(\Vd\cdot\nablag\psik,\frac{1}{\mu}G\psik)_2\vert &\leq M\vert\Vd\vert_{W^{1,\infty}}\vert\B\psik\vert_2^2 \\
 &\leq C_2(K).
\end{align*}

\textit{- Control of $B_6$}
We use Proposition \ref{318} with $s=0$ to get 

\begin{align*}
\vert (S_k,\frac{1}{\mu}G\psik)_2\vert &\leq \mu M_0 \vert\B S_k\vert_2\vert\B\psik\vert_2 \\
&\leq C_2(K)
\end{align*}
where we used the control over $S_k$ given by \eqref{reste} to derive the last inequality. \\

\textit{- Synthesis} We proved that $$\frac{d}{dt}(E_k+F_k) \leq C_2(K)$$ and thus we get by integrating in time: $$\forall t\in [0;T^\epsilon],\quad(E_k+F_k)(t)\leq C_2(K)t+C_0$$ where $C_0$ only depends on the initial energy. It is easy to get that $$\vert F_k\vert \leq C_2(K)\epsilon$$ and thus $$\forall t\in [0;T^\epsilon],\quad E_k(t) \leq C_2(K)(t+\epsilon)+C_0$$ Thanks to the equivalence between $E_k$ and the initial energy for this problem $\mathcal{E}^k_\sigma$, introduced for Theorem \ref{uniform_result} by \eqref{energie_theoreme},  proved in Remark \ref{energy_size}, we conclude to the desired result : 
$$\forall t\in [0;T^\epsilon],\forall 0\leq k\leq N,\quad\vert \zetak\vert_{H^1_\sigma}+\vert \B\psik\vert_2  \leq C_2(K)(t+\epsilon)+C_0$$ \hfill$\qquad \Box  $
\end{proof}
The Proposition \ref{lemma_control} gives the desired estimate of the form \eqref{desired} for time derivatives (more precisely, for $\vert \zetak\vert_{H^1_\sigma}+\vert\B\psik\vert_2$). We want to recover the same estimate for space derivatives, using directly the equation \eqref{ww_equation2}. To this purpose, one has to make precisely the link between the $H^1$ norm of $\B\psik$ and the $H^1_\sigma$ norm of $\zeta_{(k+1)}$. This is the point of the following Lemma.

\begin{lemma}\label{lemma_recover}
For all $0\leq\vert\alpha\vert\leq N$, one has:
$$\vert \mathfrak{P}\psi_{(\alpha,k)}\vert_2 + \vert \zeta_{(\alpha,k)}\vert_{H^1_\sigma} \leq C_1(K)(t+\epsilon)+C_0$$ where $C_0$ only depends on the initial energy of the unknowns, for all $0\leq k\leq N-\vert\alpha\vert$.
\end{lemma}
\begin{proof}[Proof of Lemma \ref{lemma_recover}]
The proof is done by induction on $\vert\alpha\vert$. For $\vert\alpha\vert=0$ it is the Lemma \ref{lemma_control}. \\

Assume that the result is true for $\vert\alpha\vert-1$, with $\vert\alpha\vert \geq 1$. Let $0\leq k\leq N-\vert\alpha\vert$. We start to give an evolution equation in term of unknowns $\zetaak$ and $\psiak$ :

\begin{align}
\begin{cases}
\ds \partial_t\zetaak+\Vd\cdot\nablag\zetaak-\frac{1}{\mu\epsilon}G\psiak-\frac{1}{\mu\epsilon}\sum dG(\partial^{l_1}\beta b,...,\partial^{l_i}\beta b,\partial^{m_1}\epsilon\zeta,...,\partial^{m_j}\epsilon\zeta) \partial^\delta \psi \\
\ds \partial_t\psiak+\Vd\cdot\nablag\psiak+\frac{1}{\epsilon}\rt \zetaak-\frac{1}{\epsilon B_0}\nablag\cdot\K(\sqrt{\mu}\nablag\zeta)\nablag\zetaak + \K_{(\alpha)}[\epsilon\sqrt{\mu}\nablag\zeta]\zeta_{\langle\widecheck{(\alpha,k)}\rangle}=S_k \label{quasilinear2}
\end{cases}
\end{align}
where the summation in the first equation is over the index $(i,j,l_1,...,l_j,m_1,...,m_j,\delta)$ satisfying $$1\leq i+j,\vert l_1+...+l_i+m_1+...+m_j\vert+\delta =\vert \alpha\vert+k.$$ We used the notation $$(\zeta_{\langle\widecheck{(\alpha,k)}\rangle} = (\zeta_{(\widecheck{\alpha}^1)},...,\zeta_{(\widecheck{\alpha}^d)},\zeta_{(\alpha,k-1)})$$ with $\widecheck{\alpha}^j=\alpha-e_j$, where $e_j$ denote the unit vector in the $j$-th direction of $\R^d$, and where $$\vert \B S_k\vert_2\leq C_1(K).$$ This system is very similar to the evolution equation \eqref{quasilinear} in term of unknowns $\zetak,\psik$, except that the space derivatives of the bottom $b$ appears from the derivation of $G\psi$, and the same goes for the space derivatives of $\zeta$ in the derivation of the surface tension term $\kappa_\gamma$. Now, using Proposition \ref{equivanorme}, one can write : 

\begin{align*}
\vert \B\psiak\vert_2^2+\vert\zetaak\vert_{H^1_\sigma}^2 \leq C_1(K) \frac{1}{\mu} \big(G\psiak,\psiak)_2+ (\rt \zetaak-\frac{1}{B_0}\nablag\K(\sqrt{\mu}\epsilon\nablag\zeta)\nablag\zetaak,\zetaak)_2\big)
\end{align*}
We now use the evolution equation \eqref{quasilinear2} in order to express $\psiak$ and $\zetaak$ with respect to time derivatives plus over terms of size $\epsilon$: 

\begin{align}
\vert \B\psiak\vert_2^2+\vert\zetaak\vert_{H^1_\sigma}^2 &\leq C_1(K) \bigg( \epsilon\partial_t\zetaak+\epsilon\Vd\cdot\nablag\zetaak\nonumber\\
&-\frac{1}{\mu}\sum dG(\partial^{l_1}\beta b,...,\partial^{l_i}\beta b,\partial^{m_1}\epsilon\zeta,...,\partial^{m_j}\epsilon\zeta) \partial^\delta \psi,\psiak \bigg)_2\nonumber\\
&+\bigg(-\epsilon\partial_t\psiak-\epsilon\Vd\cdot\nablag\psik-\frac{\epsilon}{B_0}\K(\alpha,k)[\sqrt{\mu}\nablag\zeta]\zeta_{\langle\widecheck{(\alpha,k)}\rangle}+S_k,\zetaak\bigg)_2. \label{recovery_calcul}
\end{align}

Let us control the first term of the r.h.s. of \eqref{recovery_calcul}. One has to express this term with respect to $\B\psiak$. To this purpose, we assume for convenience that $\alpha^1\neq 0$ (recall that $\vert\alpha\vert\geq 1)$. One computes :

\begin{align*}
\vert( \epsilon\partial_t\zetaak,\psiak)_2\vert &= (\partial^\alpha \zeta_{(k+1)},\psiak)_2\vert \\
&= \vert (\partial^{\alpha-e_1}\zeta_{(k+1)},\partial^{e_1}\psiak)_2\vert \\
&\leq (\vert\zeta_{(\alpha-e_1,k+1)}\vert,\vert\D\psiak\vert)_2 \\
&\leq \left\vert \frac{(1+\sqrt{\mu}\D)^{1/2}}{(1+\frac{1}{\sqrt{B_0}}\D)}(1+\frac{1}{\sqrt{B_0}}\D)\zeta_{(\alpha-e_1,k+1)}\right\vert_2 \left\vert \frac{\D}{(1+\sqrt{\mu}\D)^{1/2}}\psiak\right\vert_2
\end{align*}

Now, remark that $$ \frac{(1+\sqrt{\mu}\D)^{1/2}}{(1+\frac{1}{\sqrt{B_0}}\D)} \leq C(B_0\mu)$$ where $C(B_0\mu)$ is a constant that only depends on $B_0\mu$ (recall here that $0\leq \gamma \leq 1$ and the definition of $\D$ given in section \ref{notations}). It comes: 
\begin{align*}
( \epsilon\partial_t\zetaak,\psiak)_2 &\leq C_1(K) \vert \zeta_{(\alpha-e_1,k+1)}\vert_{H^1_\sigma}\vert \B\psiak\vert_2 \\
&\leq (C_1(K)(t+\epsilon)+C_0)\vert \B\psiak\vert_2
\end{align*}
where we used the induction assumption to control $\vert \zeta_{(\alpha-e_1,k+1)}\vert_{H^1_\sigma}$, since $\vert\alpha-e_1\vert \leq \vert\alpha\vert-1$. \\

For the control of the third term of the r.h.s. of \eqref{recovery_calcul}, one can prove, using Proposition \ref{328} that for $j\neq 0$, one has (see also \cite{david} Chapter 4 for details): $$\frac{1}{\mu}(dG(\partial^{l_1}\beta b,...,\partial^{l_i}\beta b,\partial^{m_1}\epsilon\zeta,...,\partial^{m_j}\epsilon\zeta) \partial^\delta \psi,\psiak)_2= \epsilon R$$ with $$R\leq C_1(K).$$ For $j=0$, we use again Proposition \ref{328} to write : 

\begin{align*}
\frac{1}{\mu}(dG(\partial^{l_1}\beta b,...,\partial^{l_i}\beta b) \partial^\delta \psi,\psiak)_2 &\leq C_1(K) \vert\partial^{l_1}\beta b\vert_{H^{t_0}}...\vert\partial^{l_i}\beta b\vert_{H^{t_0}}    \vert \B\partial^\delta\psi\vert_2\vert\B \psiak\vert_2 
\end{align*}
Now, recall that $b\in H^{N+1\vee t_0+1}$. In order to prove that this last term is controlled by $C_1(K)(t+\epsilon)+C_0$, one computes, using the definition of $\psia$ given by \eqref{defpsia}: 

\begin{align*}
\vert\B\partial^\delta\psi\vert_2 &\leq \vert \B\psi_{(\delta)}\vert_2+\epsilon\vert\B \underline{w}\zeta_{(\delta)}\vert_2 \\
&\leq \vert \B\psi_{(\delta)}\vert_2+\epsilon\vert \underline{w}\vert_{H^{t_0}}\vert \zeta_{(\delta)}\vert_{H^1}
\end{align*}
where we used the identity $\vert \B f\vert_2\leq \vert f\vert_{H^1}$. Now, the Proposition \ref{314} gives $\vert\underline{w}\vert \leq \mu^{3/4} C_1(K)\vert \B\psi\vert_{H^{t_0+3/2}}$. One can use the induction assumption, because the term $\psi_{(\delta)}$ only contains spatial derivatives of $\psi$ of order less than $\vert\alpha\vert-1$ due to the fact that $i\geq 1$ (the only term with spatial derivative of order $\vert \alpha\vert$ of $\psi$ that appears in the system \eqref{quasilinear2} is $G\psiak$). One gets :

$$\vert\B\partial^\delta\psi\vert_2 \leq C_1(K)(t+\epsilon)+C_0,$$ and finally 

$$(dG(\partial^{l_1}\beta b,...,\partial^{l_i}\beta b,\partial^{m_1}\epsilon\zeta,...,\partial^{m_j}\epsilon\zeta) \partial^\delta \psi,\psiak)_2 \leq \big(C_1(K)(t+\epsilon)+C_0\big)\vert\B\psi_{(\alpha,k)}\vert.$$

Now, we focus on the most difficult remaining term of \eqref{recovery_calcul}, which is $(-\epsilon\partial_t\psiak,\zetaak)_2$. One uses the definition of $\psiak$ given by \eqref{defpsia} $\psi_{(\alpha,k)} = \partial^\alpha(\epsilon\partial_t)^k\psi-\epsilon\underline{w}\partial^\alpha(\epsilon\partial_t)^k\zeta$ to write :

\begin{align*}
\vert(\epsilon\partial_t\psiak,\zetaak)_2 \vert&=\vert(\partial_1\psi_{(\alpha-e_1,k+1)},\zetaak)_2 + (\epsilon(\partial_1\underline{w})\zeta_{(\alpha-e_1,k+1)}-(\epsilon\partial_t)(\epsilon\underline{w})\zeta_{(\alpha,k)},\zetaak)_2\vert \\
&\leq  \vert \B \psi_{(\alpha-e_1,k+1)}\vert_2 \vert(\frac{1+\sqrt{\mu}\D)^{1/2}}{1+\frac{1}{\sqrt{B_0}}\D} (1+\frac{1}{\sqrt{B_0}}\D)\zetaak\vert_2 + \epsilon C_1(K)
\end{align*}

We now use the induction assumption to control $ \vert \B \psi_{(\alpha-e_1,k+1)}\vert_2$ by $C_1(K)(t+\epsilon)+C_0$, and the identity $$ \frac{(1+\sqrt{\mu}\D)^{1/2}}{(1+\frac{1}{\sqrt{B_0}}\D)} \leq C(B_0\mu)$$ where $C(B_0\mu)$ is a constant that only depends on $B_0\mu$. One gets:

\begin{align*}
\vert(\epsilon\partial_t\psiak,\zetaak)_2 \vert &\leq (C_1(K)(t+\epsilon)+C_0)(1+\vert \zetaak\vert_{H^1_\sigma})
\end{align*}

To conclude, we proved :

$$\vert \psiak\vert^2+\vert \zetaak\vert_{H^1_\sigma}^2 \leq (C_1(K)(t+\epsilon)+C_0)(1+\vert\psiak\vert_2+\vert \zetaak\vert_{H^1_\sigma})$$ and one can use the Young's inequality to get the desired estimate, and the Lemma \ref{lemma_recover} is true at rank $\vert\alpha\vert$.

 \hfill  $\qquad \Box $\end{proof}

According to the definition of $\mathcal{E}_\sigma^N (U)$ given by \eqref{energie_theoreme}, there is a remaining term to be controlled: it is $\vert\mathfrak{P}\psi\vert_{H^{t_0+3/2}}$. The control can be done as follows: let $r_0$ be such that $t_0+3/2\leq r_0 \leq N-1$. We have, for all $r\in\mathbb{N}^{d+1}$, $\vert r\vert = r_0$,

$$\vert\mathfrak{P}\partial^r\psi\vert_2 \leq
 \vert\mathfrak{P}\psi_{(r)}\vert_2+\epsilon\vert\mathfrak{P}\underline{w}\partial^r\zeta
\vert_2 $$
with the definition of $\psi_{(r)}$. The first term of the right hand side is controlled by previous estimates. For the second term, we use again the same technique as in the control of $A_4$ to write  
\begin{align*}\vert\mathfrak{P}\partial^r\psi\vert_2&\leq C_1(K)t+C_0+\epsilon\mu^{-1/4}\vert\underline{w}\partial^r\zeta\vert_{H^{1/2}} \\
&\leq C_1(K)t+C_0+\epsilon\mu^{-1/4}\vert\underline{w}\vert_{H^{t_0+1/2}}\vert\zeta\vert_{H^N} \\
&\leq C_1(K)t+C_0+\epsilon\mu^{-1/4} \mu^{3/4}\vert\mathfrak{P}\psi\vert_{H^{t_0+1}} \\
&\leq C_1(K)t+C_0+\epsilon C_1(K).
\end{align*}
 The result then comes from the identity $$\vert\mathfrak{P}\psi\vert_{H^{t_0+3/2}} \leq \vert\mathfrak{P}\psi\vert_{H^{r_0}} \lesssim \vert\mathfrak{P}\psi\vert_2 + \underset{\vert r\vert = r_0}{\sum}\vert\mathfrak{P}\partial^r\psi\vert_2.$$
 Finally, we prove $$\mathcal{E}^N(U)(t)\leq C_0+C_1(K)(t+\epsilon)$$ for all $t$ in $[0;T_\epsilon]$, which closes the proof of Proposition \ref{main_result}. \hfill \end{proof}$\qquad \Box $

We can now prove Theorem \ref{uniform_result} by constructing an existence time for all solutions $U^\epsilon$ of the system $\eqref{ww_equation2}$, which does not depend on $\epsilon$. 
\begin{proof}[Proof of Theorem \ref{uniform_result}]
We define $$\epsilon_0 = \frac{1}{2C_1(2C_0)}.$$ Let fix a $\epsilon \leq \epsilon_0$. Let us consider
\begin{equation*}
T_{\epsilon}^* = \underset{t>0}{\sup}\lbrace t,U^{\epsilon}\text{ exists on } \left[ 0,t\right]\text{ and } \mathcal{E}^N(U^{\epsilon})(t) \leq 2C_0,1+\epsilon\zeta(t)-\beta b\geq h_{\min}/2,\rt(t)\geq a_0/2\text{ on }\left[ 0,t\right]\rbrace.
\end{equation*}
We know that $T_{\epsilon}^*$ exists and that $U^{\epsilon}$, solution to (\ref{ww_equation1}) exists on $\left[ 0,T_{\epsilon}^*\right]$. The Proposition \ref{main_result} gives the following estimate:
\begin{equation*}
\mathcal{E}_{\alpha}(U^{\epsilon})(t) \leq C_1(K)(t+\epsilon) +C_0\quad\forall t\in\left[ 0,T_{\epsilon}^* \right].
\end{equation*}
We then consider 
\begin{equation*}
T_0 = \frac{1}{2C_1(2C_0)}\inf\lbrace1,h_{\min}\rbrace.
\end{equation*}
Let us show that $T_0\leq T_{\epsilon}^*$. Suppose $T_{\epsilon}^* < T_0$. First of all, let us prove that the condition over the height $1+\epsilon\zeta-\beta b$ is satisfied. One can write for all $0\leq t\leq T_0$\begin{align*}
(1+\epsilon\zeta-\beta b)(t) &= (1+\epsilon\zeta-\beta b)(0)+\int_0^t \partial_t (1+\epsilon\zeta-\beta b)(s)ds \\
&\geq (1+\epsilon\zeta-\beta b)(0)-t\underset{s\in[0;t]}{\sup} \vert\partial_t( 1+\epsilon\zeta-\beta b)\vert_{L^\infty(\R^d)} \\
&\geq (1+\epsilon\zeta-\beta b)(0)-tC_1(K) \\
&\geq h_{\min}-T_0C(2C_0) \\
&\geq \frac{h_{\min}}{2}.
\end{align*}One can do the same for the Rayleigh-Taylor condition $\rt(t)\geq \frac{a_0}{2}$. Now, for the energy condition, we would have for all $t\in\left[ 0,T_{\epsilon}^*\right]$:
\begin{align*}
\mathcal{E}_{\alpha}(U^{\epsilon})(t) &\leq C_0+C_1(K)(t+\epsilon)
\\ &\leq C_0 + C_1(2C_0)(T_{\epsilon}^*+\frac{1}{2C_1(2C_0)})
\\&< C_0 + C_1(2C_0)(T_0+\frac{1}{2C_1(2C_0)})
\\&< 2C_0
\end{align*}
using the monotony of $C$ and the definition of $T_{\epsilon}^*$. We can therefore continue the solution to an interval  $\left[ 0,\widetilde{T_{\epsilon}^*}\right]$ such that
 \begin{equation*} \mathcal{E}_{\alpha}(U^{\epsilon})(t) \leq 2C_0,\quad\quad\forall t\in \left[ 0,\widetilde{T_{\epsilon}^*} \right] \end{equation*} which contradicts the definition of $T_{\epsilon}^*$. We then have $T_0\leq T_{\epsilon}^*$. \\\\Conclusion: the solution $U^\epsilon$ exists on the time interval $[0;T_0]$. \end{proof}
  $\qquad \Box $ 

\subsection{Shallow Water limit}\label{shallow_section}
We discuss here the size of the existence time for solutions of the Water Waves equation \eqref{ww_equation3} when the shallowness parameter $\mu$ goes to zero. This regime corresponds to the Shallow Water model:

\begin{align}\begin{cases}\label{shallow_eq}
\partial_t\zeta+\nablag\cdot(h\Vu)=0 \\
\partial_t\Vu+\nablag\zeta+\epsilon(\Vu\cdot\nablag)\Vu = 0
\end{cases}\end{align}
where $h=1+\epsilon\zeta-\beta b$ is the height of the Water. Using the existence time given by Theorem \ref{uniform_result}, one can deduce easily the following long time existence result for the Shallow Water problem, proved in \cite{bresch_metivier}:

\begin{theorem}
Let $t_0>d/2$. Let $(\Vu_0,\zeta_0)\in H^{t_0+1}(\R^d)^{d+1}$. Then, there exists $T>0$ and a unique solution $(\Vu,\zeta)\in C([0;\frac{T}{\epsilon}];H^{t_0+1}(\R^d)^{d+1})$ to the Shallow Water equation \eqref{shallow_eq} with initial condition $(\Vu_0,\zeta_0)$, with 
$$\frac{1}{T} = C_1,\quad\text{ and } \underset{t\in [0;\frac{T}{\epsilon}]}{\sup} \vert (\zeta , \Vu)(t)\vert_{H^{t_0+1}(\R^d)^{d+1}}=C_2$$ where $C_i = C(\frac{1}{h_{\min}},\vert (\zeta_0,\Vu_0)\vert_{H^{t_0+1}(\R^d)^{d+1}})$ is a non decreasing function of its arguments.
\end{theorem}
The result of \cite{bresch_metivier} can therefore be understood as a particular endpoint of our main result. It is important to notice that the time existence provided here for the solutions of Shallow Water equations does not depend on the bathymetric parameter $\beta$.

\begin{proof}

Let us fix $\epsilon >0$ during all the proof. We also set $\mu B_0 = 1$. We recall that Theorem \ref{uniform_result} gives a solution $U^\mu = (\zeta^\mu,\psi^\mu)$ to the Water Waves equation \eqref{ww_equation3}, on a time interval $[0;\frac{T}{\epsilon}]$ with\begin{equation}\label{shallow_bound}\frac{1}{T}=C_1\quad\text{ and }\quad\underset{t\in [0;\frac{T}{\epsilon}]}{\sup} \sum_{\vert (\alpha,k)\vert\leq N} \vert \B\psi_{(\alpha,k)}^\mu\vert_2+\vert \zeta_{(\alpha,k)}^\mu\vert_2 =C_2\end{equation} where $C_i = C(\mathcal{E}_\sigma^N(U^0),\frac{1}{h_{\min}},\frac{1}{a_0},\vert b\vert_{H^{N+1\vee t_0+1}},\mu B_0)$ is a non decreasing function of its arguments. \\

We now need an asymptotic expansion of $\nablag\psi^\mu$ and $G\psi^\mu$ with respect to the vertical mean of the horizontal component of the velocity $V = \nablag\Phi$ in shallow water:

\begin{equation}\label{asymptotics}
G\psi^\mu = -\mu\nablag\cdot(h^\mu\Vu^\mu)\quad \text{ and }\quad \nablag\psi^\mu = \Vu^\mu +\mu R,
\end{equation}
with $\vert R\vert_{H^{t_0+1}} \leq C_2$ (recall that $N\geq t_0+t_0\vee 2+3/2$ with $t_0>d/2$) and $\ds \Vu^\mu = \frac{1}{h^\mu} \int_{-1+\beta b}^{\epsilon\zeta} V^\mu (z)dz$. For a complete proof of this latter result, see \cite{david} Chapter 3. Now, we take the first equation of \eqref{ww_equation3}, and the gradient of the second equation of \eqref{ww_equation3} and we replace $G\psi^\mu$ and $\nablag\psi^\mu$ by the expressions given by \eqref{asymptotics}. The surface tension term is of size $\mu$ since $\frac{1}{B_0}=\mu$. One can check that we get the following equation, satisfied in the distribution sense of $D'([0;\frac{T}{\epsilon}]\times\R^d)$:

\begin{align}\begin{cases}\label{shallow_reste}
\partial_t\zeta^\mu +\nablag\cdot(h^\mu\Vu^\mu) \\
\partial_t \Vu^\mu+\epsilon \Vu^\mu\cdot\nablag\Vu^\mu +\nablag\zeta^\mu = \mu R
\end{cases}
\end{align}
with $\vert R\vert_{H^{t_0+1}} \leq C_2$.
It is then easy to show that the sequences $(\Vu^\mu)_{\mu}$ and $(\zeta^\mu)_{\mu}$ are bounded in $W^{1,\infty}([0;\frac{T}{\epsilon}];H^{t_0+1}(\R^d))$, using the bound given by \eqref{shallow_bound}. Therefore, up to a subsequence we get the weak convergence of $(\Vu^\mu,\zeta^\mu)_{\mu}$ and $(\zeta^\mu)_{\mu}$ to an element $(\Vu,\zeta)\in C([0;\frac{T}{\epsilon}];H^{t_0+1}(\R^d)^{d+1})$. The weak convergence of the linear terms of the equation \eqref{shallow_reste} does not raise any difficulty. Since $H^{t_0+1}(\R^d)$ is embedded in $C^1(\R^d)$, we also get the convergence of the non linear terms in the equation \eqref{shallow_reste}. Finally, the limit $(\Vu,\zeta)$ satisfies the Shallow Water equations \eqref{shallow_eq} in the distribution sense of $D'([0;\frac{T}{\epsilon}]\times\R^d)$. The uniqueness is classical for this kind of symmetrizable quasi-linear hyperbolic system, and is done for instance in \cite{taylor3} Chapter XVI. $\qquad \Box $ 
\end{proof}

Let us give a qualitative explanation of this latter result. We recall that Proposition \ref{energy_size} claims that $$\frac{1}{M_0}\vert\B\psi\vert_2\leq(\psi,\frac{1}{\mu}G\psi)_2 \leq M_0 \vert\B\psi\vert_2^2$$ where $$\B = \frac{\D}{(1+\sqrt{\mu}\D)^{1/2}}$$ and therefore $\frac{1}{\mu}G$ acts like an order one operator with respect to $\psi$. The idea of the proof for Theorem \ref{uniform_result} is to get a "good" energy estimate for time derivatives: $$\frac{1}{2}\vert\partial_t\B\psi\vert_2+\frac{1}{2}\vert\partial_t\zeta\vert_2 \leq C(K)t\epsilon+C_0$$ and then use the equation to recover the same estimate for space derivatives. Using the first equation $\partial_t\zeta=\frac{1}{\mu}G\psi$, and the ellipticity of the order one operator $\frac{1}{\mu}G$  would only provide us a gain of half a space derivative for $\psi$ (we already have an estimate for $\B\psi$ where $\B$ is of order $1/2$). This is why we need surface tension which provides an additional estimate for $\frac{1}{B_0}\vert\partial_t\nablag\zeta\vert_2$, which leads to recover exactly one space derivative in the estimate of $\psi$.   In the shallow water regime, using \eqref{asymptotics}, one gets that $\frac{1}{\mu}G\psi \sim -\nablag\cdot(h\nablag\psi)$ as $\mu$ goes to zero. Therefore, when $\mu$ goes to zero, $\frac{1}{\mu}G\psi$ goes to an order one operator with respect to $\Vu = \nablag\psi$. \footnote{Another way to see it is that in the limit $\mu\rightarrow 0$, $\B$ must be seen as a first order operator ($\B\sim \D$) and the control of $\B\psi$ gives the control of a full derivative of $\psi$.} Therefore, if one has a "good" estimate for time derivatives like $$\frac{1}{2}\vert\partial_t\zeta\vert_2+\frac{1}{2}\vert\partial_t \Vu\vert_2\leq Ct\epsilon+C_0$$ with $C$ independent of $\epsilon$, one can recover, using the first equation, a "good" estimate for $\nablag\cdot(\Vu)$. Using the second equation to get the same good estimate for $\mbox{curl}\Vu$, one can recover exactly one space derivative of $\Vu$ in the estimates. And therefore, we does not need surface tension in this case. Moreover, the surface tension term of size $\frac{1}{B_0} = \mu$ vanishes as $\mu$ goes to zero.

\appendix

\section{The Dirichlet Neumann Operator}\label{appendixA}
Here are for the sake of convenience some technical results about the Dirichlet Neumann operator, and its estimates in Sobolev norms. See \cite{david} Chapter 3 for complete proofs. The first two  propositions give a control of the Dirichlet-Neumann operator. 

\begin{proposition}\label{314}
Let $t_0$>d/2, $0\leq s \leq t_0+3/2$ and $(\zeta,\beta)\in H^{t_0+1}\cap H^{s+1/2}(\mathbb{R}^d)$ such that  \begin{equation*} \exists h_0>0,\forall X\in\mathbb{R}^d,  \epsilon\zeta(X)-\beta b(X) +1 \geq h_0\end{equation*} 

(1)\quad The operator $G$ maps continuously $\overset{.}H{}^{s+1/2}(\mathbb{R}^d)$ into $H{}^{s-1/2}(\mathbb{R}^d)$ and one has 
\begin{equation*}
\vert G\psi\vert_{H^{s-1/2}} \leq \mu^{3/4} M(s+1/2) \vert\mathfrak{P}\psi\vert_{H^s},
\end{equation*}
where $M(s+1/2)$ is a constant of the form $C(\frac{1}{h_0},\vert\zeta\vert_{H^{t_0+1}},\vert b\vert_{H^{t_0+1}},\vert\zeta\vert_{H^{s+1/2}},\vert b\vert_{H^{s+1/2}})$.
\\

(2)\quad The operator $G$ maps continuously $\overset{.}H{}^{s+1}(\mathbb{R}^d)$ into $H{}^{s-1/2}(\mathbb{R}^d)$ and one has 
\begin{equation*}
\vert G\psi\vert_{H^{s-1/2}} \leq \mu M(s+1) \vert\mathfrak{P}\psi\vert_{H^s+1/2},
\end{equation*}
where $M(s+1)$ is a constant of the form $C(\frac{1}{h_0},\vert\zeta\vert_{H^{t_0+1}},\vert b\vert_{H^{t_0+1}},\vert\zeta\vert_{H^{s+1}},\vert b\vert_{H^{s+1}})$.
\\

Moreover, it is possible to replace $G$ by $\underline{w}$ in the previous result, where $\underline{w} = \frac{G\psi+\epsilon\mu\nablag\zeta\cdot\nablag\psi}{1+\epsilon^2\mu\vert\nablag\zeta\vert^2}$(vertical component of the velocity $U=\nabla_{X,z}\Phi$ at the surface).
\end{proposition}

\begin{proposition}\label{318}
Let $t_0>d/2$, and $0\leq s\leq t_0+1/2$. Let also $\zeta,b\in H^{t_0+1}(\mathbb{R}^d)$ be such that  $$\exists h_0>0, \forall X\in\mathbb{R}^d, 1+\epsilon\zeta(X)-\beta b(X) \geq h_0$$ Then, for all $\psi_1$, $\psi_2\in \overset{.}H{}^{s+1/2}(\mathbb{R}^d)$, we have  $$(\Lambda^sG\psi_1,\Lambda^s\psi_2)_2 \leq\mu M_0 \vert \mathfrak{P}\psi_1\vert_{H^s}\vert \mathfrak{P}\psi_2\vert_{H^s},$$
where $M_0$ is a constant of the form $C(\frac{1}{h_0},\vert\zeta\vert_{H^{t_0+1}},\vert b\vert_{H^{t_0+1}})$.
\end{proposition}

The second result gives a control of the shape derivatives of the Dirichlet-Neumann operator. More precisely, we define the  open set  $\bold{\Gamma}\subset H^{t_0+1}(\mathbb{R}^d)^2$ as:\\
$$\bold{\Gamma} =\lbrace \Gamma=(\zeta,b)\in H^{t_0+1}(\mathbb{R}^d)^2,\quad \exists h_0>0,\forall X\in\mathbb{R}^d, \epsilon\zeta(X) +1-\beta b(X) \geq h_0\rbrace$$ and, given a $\psi\in \overset{.}H{}^{s+1/2}(\mathbb{R}^d)$, the mapping: \begin{equation}\label{mapping}G[\epsilon\cdot,\beta\cdot] : \left. \begin{array}{rcl}
&\bold{\Gamma} &\longrightarrow H^{s-1/2}(\mathbb{R}^d) \\
&\Gamma=(\zeta,b) &\longmapsto G[\epsilon\zeta,\beta b]\psi.
\end{array}\right.\end{equation} We can prove the differentiability of this mapping. The following Theorem gives a very important explicit formula for the first-order partial derivative of $G$ with respect to $\zeta$:

\begin{theorem}\label{321}
Let $t_0>d/2$. Let $\Gamma = (\zeta,b)\in \bold{\Gamma}$ and $\psi\in\overset{.}H{}^{3/2}(\R^d)$. Then, for all $h\in H^{t_0+1}(\R^d)$, one has $$dG(h)\psi = -\epsilon G(h\underline{w})-\epsilon\mu\nablag\cdot(h\Vd),$$ with $$\underline{w}=\frac{G\psi+\epsilon\mu\nablag\zeta\cdot\nablag\psi}{1+\epsilon^2\mu\modd{\nablag\zeta}},\quad\text{ and }\quad \Vd = \nablag\psi-\epsilon\underline{w}\nablag\zeta.$$
\end{theorem}

The following result gives estimates of the derivatives of the mapping \eqref{mapping}. 
\begin{proposition}\label{328}
Let $t_0$>d/2, $0\leq s \leq t_0+1/2$ and $(\zeta,\beta)\in H^{t_0+1}(\mathbb{R}^d)$ such that: \begin{equation*} \exists h_0>0,\forall X\in\mathbb{R}^d,  \epsilon\zeta(X)-\beta b(X) +1 \geq h_0\end{equation*} 
Then, for all $\psi_1,\psi_2\in \overset{.}H{}^{s+1/2}(\mathbb{R}^d)$, for all $(h,k)\in H^{t_0+1}(\mathbb{R}^d)$ one has  
\begin{equation*}
\vert (\Lambda^s d^j G(h,k)\psi_1,\Lambda^s\psi_2)\vert \leq \mu M_0 \prod_{m=1}^j \vert(\epsilon h_m,\beta k_m)\vert_{H^{t_0+1}}\vert\mathfrak{P}\psi_1\vert_s\vert\mathfrak{P}\psi_2\vert_s,
\end{equation*}
where $M_0$ is a constant of the form $C(\frac{1}{h_0},\vert\zeta\vert_{H^{t_0+1}},\vert b\vert_{H^{t_0+1}})$.
\end{proposition}

The following Proposition gives the same type of estimate that the previous one:

\begin{proposition}\label{328b}
Let $t_0>d/2$ and $(\zeta,b)\in H^{t_0+1}$ be such that \begin{equation*} \exists h_0>0,\forall X\in\mathbb{R}^d,  \epsilon\zeta(X)-\beta b(X) +1 \geq h_0.\end{equation*} Then, for all $0\leq s\leq t_0+1/2$, $$\vert d^j G(h,k)\psi\vert_{H^{s-1/2}} \leq M_0 \mu^{3/4} \prod_{m=1}^j \vert (\epsilon h_m,\beta k_m)\vert_{H^{t_0+1}} \vert \B\psi\vert_{H^s}$$
\end{proposition}
We  need the following commutator estimate:

\begin{proposition}\label{329}
Let $t_0>d/2$ and $\zeta, b \in H^{t_0+2}(\mathbb{R}^d)$ such that:  \begin{equation*} \exists h_0>0,\forall X\in\mathbb{R}^d,  \epsilon\zeta(X)-\beta b(X) +1 \geq h_0\end{equation*} 
For all $\underline{V}\in H^{t_0+1}(\mathbb{R}^d)^2$ and $u\in H^{1/2}(\mathbb{R}^d)$, one has 

\begin{equation*}
((\underline{V}\cdot\nabla^{\gamma} u),\frac{1}{\mu}Gu)\leq M\vert\underline{V}\vert_{W^{1,\infty}}\vert \mathfrak{P} u\vert_2^2,
\end{equation*}
where $M$ is a constant of the form $C(\frac{1}{h_0},\vert\zeta\vert_{H^{t_0+2}},\vert b\vert_{H^{t_0+2}})$.
\end{proposition}

\bibliographystyle{plain}
\bibliography{these} 
The author has been partially funded by the ANR  project Dyficolti ANR-13-BS01-0003-01.
\end{document}